\documentclass{article}

\usepackage{microtype}
\usepackage{graphicx}
\usepackage{subfigure}
\usepackage{booktabs} % for professional tables

\usepackage{hyperref}

\usepackage[accepted]{icml2018}

\usepackage{pifont}

\usepackage{epsfig}
\usepackage{amssymb}
\usepackage{amsmath}
\usepackage{amsthm}
\usepackage{amsfonts}
\usepackage{bbding}
\usepackage{array}
\usepackage{caption,tabularx,booktabs}
\usepackage{arydshln}
\usepackage{xargs}                      % Use more than one optional parameter in a new commands
\usepackage[pdftex,dvipsnames]{xcolor}  % Coloured text etc.

\usepackage[colorinlistoftodos,prependcaption,textsize=tiny]{todonotes}
\newcommandx{\unsure}[2][1=]{\todo[inline,linecolor=red,backgroundcolor=red!25,bordercolor=red,#1]{#2}}
\newcommandx{\change}[2][1=]{\todo[linecolor=blue,backgroundcolor=blue!25,bordercolor=blue,#1]{#2}}
\newcommandx{\info}[2][1=]{\todo[linecolor=OliveGreen,inline,backgroundcolor=OliveGreen!25,bordercolor=OliveGreen,#1]{#2}}
\newcommandx{\improvement}[2][1=]{\todo[linecolor=Plum,inline,backgroundcolor=Plum!25,bordercolor=Plum,#1]{#2}}

\usepackage[font=small,skip=5pt]{caption}

\newtheorem{thm}{Theorem}
\newtheorem{lem}{Lemma}

\newtheorem{rem}{Remark}
\newtheorem{ass}{Assumption}

\newcommand{\ve}[2]{\langle #1, #2 \rangle}

\newcolumntype{C}[1]{>{\centering\let\newline\\\arraybackslash\hspace{0pt}}m{#1}}

\newcommand\tagthis{\addtocounter{equation}{1}\tag{\theequation}}
\DeclareMathOperator{\Prob}{\mathbb{P}}           
           
\DeclareMathOperator{\Exp}{\mathbb{E}}           % expectation
\DeclareMathOperator{\R}{\mathbb{R}} 
 
\newcommand{\eqdef}{\stackrel{\text{def}}{=}}

\makeatletter
\newcounter{subthm} 
\let\savedc@thm\c@hyp

\newcommand{\normhyp}{%
  \let\c@hyp\savedc@hyp % revert to the old one
  \renewcommand\thehyp{\arabic{hyp}}%
} 
\makeatother

\makeatletter
\newcounter{subass} 
\let\savedc@ass\c@hyp

\makeatother

\icmltitlerunning{SGD and Hogwild! Convergence Without the Bounded Gradients Assumption}

\begin{document}

\twocolumn[
% \icmltitle{Gradients Needs Not Be Bounded for SGD to Converge}
\icmltitle{SGD and Hogwild! Convergence Without the Bounded Gradients Assumption}




\begin{icmlauthorlist}
\icmlauthor{Lam M. Nguyen}{af1,af4}
\icmlauthor{Phuong Ha Nguyen}{af2}
\icmlauthor{Marten van Dijk}{af2}
\icmlauthor{Peter Richt\'{a}rik}{af3}
\icmlauthor{Katya Scheinberg}{af1}
\icmlauthor{Martin Tak\'{a}\v{c}}{af1}
\end{icmlauthorlist}

\icmlaffiliation{af1}{Department of Industrial and Systems Engineering, Lehigh University, USA.}
\icmlaffiliation{af4}{IBM Thomas J. Watson Research Center, USA.}
\icmlaffiliation{af2}{Department of Electrical and Computer Engineering, University of Connecticut, USA.}
\icmlaffiliation{af3}{KAUST, KSA --- Edinburgh, UK ---  MIPT, Russia}

\icmlcorrespondingauthor{Lam M. Nguyen}{LamNguyen.MLTD@gmail.com}
\icmlcorrespondingauthor{Phuong Ha Nguyen}{phuongha.ntu@gmail.com}
\icmlcorrespondingauthor{Marten van Dijk}{marten.van$\_$dijk@uconn.edu}
\icmlcorrespondingauthor{Peter Richt\'{a}rik}{Peter.Richtarik@ed.ac.uk}
\icmlcorrespondingauthor{Katya Scheinberg}{katyas@lehigh.edu}
\icmlcorrespondingauthor{Martin Tak\'{a}\v{c}}{Takac.MT@gmail.com}

\icmlkeywords{Machine Learning, ICML}

\vskip 0.3in
]



\printAffiliationsAndNotice{\icmlEqualContribution} % otherwise use the standard text.

\begin{abstract}
Stochastic gradient descent (SGD) is the optimization algorithm of choice in many machine learning applications such as regularized empirical risk minimization and training deep neural networks. The classical convergence analysis of SGD is carried out under the assumption  that the norm of the stochastic gradient is uniformly bounded. While this might hold for some loss functions, it is always violated for  cases where the objective  function is strongly convex.  In \cite{bottou2016optimization}, a new analysis of convergence of SGD is performed under the assumption that stochastic gradients are bounded with respect to the true gradient norm. Here we show that for stochastic problems arising in machine learning
such bound always holds; and we also propose an alternative convergence analysis of SGD with diminishing learning rate regime, which results in more relaxed conditions than those in \cite{bottou2016optimization}. We then move on the asynchronous parallel setting, and prove convergence of Hogwild! algorithm in the same regime, obtaining the first convergence results for this method in the case of diminished learning rate.
\end{abstract}

\section{Introduction}\label{intro}

We are interested in solving the following stochastic optimization problem
\begin{align*}
\min_{w \in \mathbb{R}^d} \left\{ F(w) = \mathbb{E} [ f(w;\xi) ] \right\}, \tagthis \label{main_prob_expected_risk}  
\end{align*}
where $\xi$ is a random variable  obeying some distribution. 

In the case of empirical risk minimization with a training set $\{(x_i,y_i)\}_{i=1}^n$, $\xi_i$ is a random variable that is defined 
 by a single random sample  $(x,y)$  pulled uniformly from the training set. Then,  by defining  $f_i(w) := f(w;\xi_{i})$, empirical risk minimization reduces to  
\begin{gather}\label{main_prob}
\min_{w \in \mathbb{R}^d} \left\{ F(w) = \frac{1}{n} \sum_{i=1}^n f_i(w) \right\}.  
\end{gather}

Problem  \eqref{main_prob} arises frequently in supervised learning applications \cite{ESL}. For  a wide range of applications, such as linear regression and logistic regression, the objective function $F$ is strongly convex and each $f_i$, $i \in [n]$, is convex and has Lipschitz continuous gradients (with Lipschitz constant $L$).   Given a training set $\{(x_i,y_i)\}_{i=1}^n$ with $x_i \in\R^{d}, y_i\in\R$, the $\ell_2$-regularized least squares regression model, for example, is written as \eqref{main_prob} with $f_i(w)\eqdef (\langle x_i,w \rangle -y_i)^2 + \frac{\lambda}{2} \|w\|^2$. The $\ell_2$-regularized logistic regression  for binary classification is written with $f_i(w) \eqdef \log (1+\exp(-y_i \langle x_i,w \rangle)) + \frac{\lambda}{2}\|w\|^2$, $y_i\in\{-1,1\}$. It is well established by now that solving this type of problem by gradient descent (GD) \cite{nesterov2004,Nocedal2006NO} may be prohibitively expensive and stochastic gradient descent (SGD) is thus preferable.  Recently, a class of variance reduction methods \cite{SAG,SAGA,SVRG,nguyen2017sarah} has been proposed in order to reduce the computational cost. All these methods explicitly exploit the finite sum form of \eqref{main_prob} and thus they  have some disadvantages for very large scale machine learning problems and are not applicable to \eqref{main_prob_expected_risk}. 

To apply SGD to the general form \eqref{main_prob_expected_risk}  one needs to assume existence of unbiased gradient estimators. This is usually defined as follows: 
 $$\mathbb{E_\xi}[\nabla f(w;\xi)] = \nabla F(w),$$ 
 for any fixed $w$. 
Here we make an important observation: if we view \eqref{main_prob_expected_risk} not as a general stochastic problem but as the expected risk minimization problem, where $\xi$ corresponds to a random data sample pulled from a distribution,
then  \eqref{main_prob_expected_risk} has  an additional key property: for each realization of the random variable $\xi$, $f(w;\xi)$ is a convex function with Lipschitz continuous gradients.
Notice that traditional analysis of SGD for general stochastic problem of the form  \eqref{main_prob_expected_risk} does not make any assumptions on individual function realizations. 
In this paper we derive convergence properties for SGD applied to \eqref{main_prob_expected_risk} with these additional assumptions on $f(w;\xi)$ and also extend to the case when  $f(w;\xi)$  are not necessarily convex.

%In this paper, we consider the classical stochastic gradient descent (SGD)~\footnote{We mark here that even though stochastic gradient is referred to as SG in literature, the term stochastic gradient descent (SGD) has been widely used in many important works of large-scale learning} \cite{RM1951} algorithm, which does not depend on the size of data $n$. Instead of evaluating full gradient, SGD evaluates only a single component gradient at each step.  
%where $i_t \in [n]$, $t \geq 0$, is chosen uniformly at random, as follows
%\begin{gather*}
%w_{t+1} = w_{t} - \eta_t \nabla f_{i_t}(w_{t}).  \tagthis \label{sgd_step}
%\end{gather*}

%Obviously, the computational cost of SGD is $n$ times cheaper than that of GD. However, we need to choose $\eta_t = \Ocal(1/t)$ and the convergence rate of SGD is slowed down to $\Ocal(1/t)$ \cite{bottou2016optimization}, which is a sublinear convergence rate. 

Regardless of the properties of $f(w;\xi)$ we assume that $F$ in \eqref{main_prob_expected_risk}  is strongly convex. We define the (unique) optimal solution of $F$ as $w_{*}$.  
\begin{ass}[$\mu$-strongly convex]
\label{ass_stronglyconvex}
The objective function $F: \mathbb{R}^d \to \mathbb{R}$ is a $\mu$-strongly convex, i.e., there exists a constant $\mu > 0$ such that $\forall w,w' \in \mathbb{R}^d$, 
\begin{gather*}
F(w) - F(w') \geq \langle \nabla F(w'), (w - w') \rangle + \frac{\mu}{2}\|w - w'\|^2. \tagthis\label{eq:stronglyconvex_00}
\end{gather*}
\end{ass}
It is well-known in literature \cite{nesterov2004,bottou2016optimization} that Assumption \ref{ass_stronglyconvex} implies
\begin{align*}
2\mu [ F(w) - F(w_{*}) ] \leq \| \nabla F(w) \|^2 \ , \ \forall w \in \mathbb{R}^d. \tagthis\label{eq:stronglyconvex}
\end{align*}

The classical theoretical analysis of SGD assumes that the  \textit{ stochastic gradients are uniformly bounded}, i.e. there exists a finite (fixed) constant $\sigma<\infty$, such that
\begin{align}\label{bounded_grad_ass}
\mathbb{E}[\| \nabla f(w; \xi) \|^2] \leq \sigma^2 \ , \ \forall w \in \mathbb{R}^d
\end{align}
(see e.g. \cite{ShalevShwartz2007,Nemirovski2009,Hogwild,Hazan2014,Rakhlin2012}, etc.). However, this assumption is clearly false if $F$ is strongly convex. Specifically, under this assumption together with strong convexity,  $\forall w \in \mathbb{R}^d$, we have
\begin{align*}
2\mu [ F(w) - F(w_{*}) ] & \overset{\eqref{eq:stronglyconvex}}{\leq} \| \nabla F(w) \|^2 = \left\| \mathbb{E}[ \nabla f(w; \xi) ] \right\|^2 \\ &\leq \mathbb{E}[ \| \nabla f(w; \xi) \|^2 ] \overset{\eqref{bounded_grad_ass}}{\leq} \sigma^2. 
\end{align*} 
Hence, 
\begin{align*}
F(w) \leq \frac{\sigma^2}{2\mu} + F(w_{*}) \ , \ \forall w \in \mathbb{R}^d. 
\end{align*}
On the other hand strong convexity and $\nabla F(w_{*})=0$ imply
\begin{align*}
F(w) \geq {\mu} \|w-w_*\|^2+ F(w_{*}) \ , \ \forall w \in \mathbb{R}^d. 
\end{align*}
The last two inequalities are clearly in  contradiction with each other for sufficiently large  $\|w-w_*\|^2$. 

Let us consider the following example: $f_1(w)= \frac{1}{2} w^2$ and $f_2(w)=w$ with $F(w)=\frac{1}{2}(f_1(w)+f_2(w))$. Note that $F$ is strongly convex, while individual realizations are not necessarily so.  Let $w_0=0$, for any number $t\geq 0$, with probability $\frac{1}{2^t}$ the steps of SGD algorithm for all $i< t$ are $w_{i+1}=w_i-\eta_i$. This implies that
$w_t=-\sum_{i=1}^t \eta_i$ and  since $\sum_{i=1}^\infty\eta_i =\infty$ then $|w_t|$ can be arbitrarily large for large enough $t$ with probability $\frac{1}{2^t}$. Noting that for this example, $\mathbb{E}[\| \nabla f(w_t; \xi) \|^2] =\frac{1}{2}w^2_t+\frac{1}{2}$, we see that $\mathbb{E}[\| \nabla f(w_t; \xi) \|^2]$ can also be arbitrarily large. 

%Instead of assuming that \eqref{bounded_grad_ass} is bounded for all $w$, it is enough to  assume $\mathbb{E}[\| \nabla f(w_t; \xi_t) \|^2] \leq \sigma^2$, where $w_t$, $t \geq 0$, are the iterates generated by the algorithm. From our analysis above, it clearly implies the assumption that the iterates of the  algorithm remain in a bounded region around $w_*$. This condition is impossible to guarantee with probability one for the classical stochastic gradient method. 

%Equation \eqref{bounded_grad_ass} makes sense only when $\sigma \to \infty$. However, the theoretical analysis must depend on $\sigma$ in this situation and would definitely affect the convergence results. Many previous results simply ignored the dependency of $\sigma$ in the convergence rate. 

%For other cases, which are not strongly convex, this bounded gradient assumption also makes the number of functions that may be useful in applications become limited, especially when people assume that the bounded constant $\sigma$ is small. To expand the choice of functions, the value of $\sigma$ should become large (or very large). Thus, it would affect negatively the convergence results of the algorithm since the rate of convergence must also depend on $\sigma$. The theoretical results may become worse when $\sigma \to \infty$. 

%Although people may suppose that there exists some compact convex set containing all iterations of SGD algorithm (see e.g. \cite{Cohen2016}), there is theoretically no guarantee that all of those iterations are still in that set. 

Recently, in the review paper \cite{bottou2016optimization}, convergence of SGD for general stochastic optimization problem was analyzed under the following assumption:  there exist constants $M$ and $N$ such that $\mathbb{E}[\| \nabla f(w_t; \xi_t) \|^2] \leq M \| \nabla F(w_t) \|^2 + N$, where $w_t$, $t \geq 0$, are generated by the algorithm. This assumption does not contradict strong convexity, however, in general, constants $M$ and $N$ are unknown, while $M$ is used to determine the learning rate $\eta_t$ \cite{bottou2016optimization}.  In addition, the rate of convergence of the SGD algorithm depends  on $M$ and $N$. In this paper we show that under the smoothness assumption on individual realizations  $f(w,\xi)$ it is possible to derive the bound $\mathbb{E}[\| \nabla f(w; \xi) \|^2] \leq M_0 [F(w) - F(w_*)] + N$  with specific values of $M_0$, and $N$ for $\forall w \in \mathbb{R}^d$, which in turn  implies the bound $\mathbb{E}[\| \nabla f(w; \xi) \|^2] \leq M \| \nabla F(w) \|^2 + N$ with specific $M$, by strong convexity of $F$. We also note that, in \cite{Bach_NIPS2011}, the convergence of SGD without bounded gradient assumption is studied. We then provide an alternative  convergence analysis for SGD which shows convergence in expectation  with a bound on learning rate  which is  larger than that in \cite{bottou2016optimization,Bach_NIPS2011} by a factor of  $L/\mu$. We then use the  new framework for the convergence analysis of SGD to analyze an asynchronous stochastic gradient method. 

%In this paper, we also provide the sufficient conditions (based on learning rate) for almost sure (w.p.1.)\footnote{With probability 1.} convergence of SGD. We provide the new framework for the convergence analysis of SGD, which may be able to use for the analysis of other stochastic gradient algorithms related to SGD without requiring bounded gradient assumption in the strongly convex case. 

In~\cite{Hogwild}, an asynchronous stochastic optimization method called Hogwild! was proposed.  Hogwild! algorithm is a parallel version of SGD, where each processor applies SGD steps independently of the other processors to the solution $w$ which is  shared by all processors. Thus, each processor computes a stochastic gradient and updates $w$ without "locking" the memory containing $w$, meaning that multiple processors are able to update $w$ at the same time.  This approach leads to much better scaling of parallel SGD algorithm than a synchoronous version, but the analysis of this method is more complex.   In ~\cite{Hogwild,ManiaPanPapailiopoulosEtAl2015,DeSaZhangOlukotunEtAl2015} various variants of   Hogwild! with a fixed step size are analyzed under the assumption that the gradients are bounded as in  (\ref{bounded_grad_ass}). In this paper, we  extend our analysis of SGD to provide analysis  of Hogwild! with diminishing step sizes and without the   assumption on bounded gradients.

In a recent technical report \cite{Leblond2018}  Hogwild! with fixed step size is analyzed without the bounded gradient assumption. We note that SGD with fixed step size only converges  to a neighborhood of the optimal solution, while by analyzing the diminishing step size variant we are able to show convergence to the \textit{optimal solution} with probability one. Both in \cite{Leblond2018}  and in this paper, the version of Hogwild! with inconsistent reads and writes is considered. 

\subsection{Contribution}

We provide a new framework for the analysis of stochastic gradient algorithms in the strongly convex case under the condition of Lipschitz continuity of the individual function realizations, but {\bf without requiring any bounds on the stochastic gradients}.
Within this framework we have the following contributions: 
\begin{itemize}
\item 
%(1) 
We prove the almost sure (w.p.1) convergence of SGD with diminishing step size. Our analysis provides a larger bound on the possible initial  step size when compared to any previous analysis of convergence in expectation for SGD. 

\item 
%(2) 
We introduce a general recurrence for vector updates which has as its special cases (a) Hogwild! algorithm with diminishing step sizes, where each update involves all non-zero entries of the computed gradient, and (b) a position-based updating algorithm where each update corresponds to only one uniformly selected non-zero entry of the computed gradient.

\item 
%(3)
We analyze this general recurrence under inconsistent vector reads from and vector writes to shared memory (where individual vector entry reads and writes are atomic in that they cannot be interrupted by writes to the same entry) assuming that there exists a delay $\tau$ such that during the $(t+1)$-th iteration a gradient of a read vector $w$ is computed which includes the aggregate of all the updates up to and including those made during the $(t-\tau)$-th iteration. In other words, $\tau$ controls to what extend past updates influence the shared memory.
\begin{itemize}
\item 
%(3a) 
Our upper bound for the expected convergence rate is sublinear, i.e., $O(1/t)$, and its precise expression allows comparison of algorithms (a) and (b) described above.
\item
%(3b)
For SGD we can improve this upper bound by a factor 2 and also show that its initial step size can be larger.

\item
%(3c)
We show that $\tau$ can be a function of $t$ as large as $\approx \sqrt{t /\ln t}$ without affecting the asymptotic behavior of the upper bound; we also determine a constant $T_0$ with the property that, for $t\geq T_0$,  higher order terms containing parameter $\tau$ are smaller than the leading $O(1/t)$ term. We give intuition explaining why the expected convergence rate is not more affected by $\tau$. Our experiments confirm our analysis.
\item 
%
%(3d)
We determine a constant $T_1$ with the property that, for $t\geq T_1$,  the higher order term containing parameter $\|w_0-w_*\|^2$ is smaller than the leading $O(1/t)$ term.
\end{itemize}
\item 
%
%(4) 
All the above contributions generalize to the non-convex setting where we do not need to assume that the component functions $f(w;\xi)$ are convex in $w$.
\end{itemize}

%2) Our analysis is better compared to 2018 framework in that we have no dependence on tau and sparsity for the leading term (in our attempt to use the 2018 framework -- since they have not done this for diminishing step size -- we do see such dependency).

%= All results hold in the not fixed sum case

\subsection{Organization}

We analyse the convergence rate of SGD  in Section \ref{analysis} and introduce the general recursion and its analysis in Section \ref{general}. Experiments are reported in Section \ref{sec_experiments}.

\section{New Framework for Convergence Analysis of SGD}
\label{analysis}

We introduce SGD algorithm in Algorithm \ref{sgd_algorithm}. 
%\begin{algorithm}[h!]
%   \caption{Stochastic Gradient Descent (SGD) Method}
%   \label{sgd_algorithm}
%\begin{algorithmic}
%   \STATE {\bfseries Initialize:} $w_0$
%   \STATE {\bfseries Iterate:}
%   \FOR{$t=0,1,2,\dots$}
%   \STATE Choose a stepsize $\eta_t > 0$. 
%   \STATE Generate $i_t$ uniformly at random from $\{1,\dots,n\}$. 
%   \STATE $w_{t+1} = w_{t} - \eta_t \nabla f_{i_t}(w_{t})$. 
%   \ENDFOR
%\end{algorithmic}
%\end{algorithm}
\begin{algorithm}[h]
   \caption{Stochastic Gradient Descent (SGD) Method}
   \label{sgd_algorithm}
\begin{algorithmic}
   \STATE {\bfseries Initialize:} $w_0$
   \STATE {\bfseries Iterate:}
   \FOR{$t=0,1,2,\dots$}
  \STATE Choose a step size (i.e., learning rate) $\eta_t>0$. 
  \STATE Generate a random variable $\xi_t$.
  \STATE Compute a stochastic gradient $\nabla f(w_{t};\xi_{t}).$
   \STATE Update the new iterate $w_{t+1} = w_{t} - \eta_t \nabla f(w_{t};\xi_{t})$.
   \ENDFOR
\end{algorithmic}
\end{algorithm}
%
%Note that $\mathcal{F}_{t} = \sigma(w_{0},\xi_{0},\dots,\xi_{t-1})$ is the $\sigma$-algebra generated by $w_{0},\xi_{0},\dots,\xi_{t-1}$, i.e., $\mathcal{F}_{t}$ contains all the information of $w_{0},\dots,w_{t}$. 

The sequence of random variables $\{\xi_t\}_{t \geq 0}$ is assumed to  be i.i.d.\footnote{Independent and identically distributed.}
%$w \in \mathbb{R}^d$, $$\mathbb{E}[\nabla f(w;\xi_{i})] = \nabla F(w).$$ 
Let us introduce our key assumption that each realization $\nabla f(w;\xi)$ is an $L$-smooth function. 

%\begin{ass}[$L$-smooth]
%\label{ass_smooth}
%Each $f_i: \mathbb{R}^d \to \mathbb{R}$, $i \in \{1,\dots,n\}$, is $L$-smooth, i.e., there exists a constant $L > 0$ such that
%\begin{align*}
%\| \nabla f_i(w) - \nabla f_i(w') \| \leq L \| w - w' \|, \ \forall w,w' \in \mathbb{R}^d. \tagthis\label{eq:Lsmooth_basic}
%\end{align*} 
%\end{ass}

\begin{ass}[$L$-smooth]
\label{ass_smooth}
$f(w;\xi)$ is $L$-smooth for every realization of $\xi$, i.e., there exists a constant $L > 0$ such that, $\forall w,w' \in \mathbb{R}^d$, 
\begin{align*}
\| \nabla f(w;\xi) - \nabla f(w';\xi) \| \leq L \| w - w' \|. \tagthis\label{eq:Lsmooth_basic}
\end{align*} 
\end{ass}

Assumption \ref{ass_smooth} implies that $F$ is also $L$-smooth. Then, by the property of $L$-smooth function (in \cite{nesterov2004}), we have, $\forall w, w' \in \mathbb{R}^d$, 
\begin{align*}
F(w) &\leq F(w') + \langle \nabla F(w'),(w-w') \rangle + \frac{L}{2}\|w-w'\|^2. \tagthis\label{eq:Lsmooth}
\end{align*}

The following additional convexity assumption can be made, as it holds for many problems arising in machine learning.
%\begin{ass}\label{ass_convex}
%Each function $f_i: \mathbb{R}^d \to \mathbb{R}$, $i \in \setn$, is convex, i.e.,
%\begin{gather*}
%f_i(w)  - f_i(w') \geq \langle \nabla f_i(w'),(w - w') \rangle, \quad \forall i\in\setn. 
%\end{gather*}
%\end{ass}

\begin{ass}\label{ass_convex}
$f(w;\xi)$ is convex for every realization of $\xi$, i.e., $\forall w,w' \in \mathbb{R}^d$, 
\begin{gather*}
f(w;\xi)  - f(w';\xi) \geq \langle \nabla f(w';\xi),(w - w') \rangle.
\end{gather*}
\end{ass}

We first derive our analysis under Assumptions  \ref{ass_smooth}, and \ref{ass_convex} and then we derive weaker results under only
Assumption  \ref{ass_smooth}. 

\subsection{Convergence With Probability One}

As discussed in the introduction, under Assumptions \ref{ass_smooth} and \ref{ass_convex} we can now derive a bound on $\mathbb{E}\|\nabla f(w; \xi)\|^2$. 

%\textcolor{blue}{Lam: We need Assumption 3 for the proof of Lemma 3 in Appendix, which needs to be used for Lemma 1. Theorems 1 + 2 (SGD) and 3 (Hogwild!) require to use Lemma 1. If we remove Assumption 3, we would have the result in Lemma 2, which is a little worse by $\kappa$ -- which is in Section 4. Probably, in that section, we do not need to state the whole theorem (Theorems 4 and 5), but find the way to write it shorter like in Section 3.2 for Hogwild! (in the last (small) paragraph of Page 6)}. 
%
\begin{lem}\label{lem_bounded_secondmoment_04}
Let Assumptions \ref{ass_smooth} and \ref{ass_convex} hold. Then, for $\forall w \in \mathbb{R}^d$, 
\begin{gather}
\mathbb{E}[\|\nabla f(w; \xi)\|^2] \leq  4 L [ F(w) - F(w_{*}) ] + N,
\label{afsfawfwaefwea} 
\end{gather}
where $N = 2 \mathbb{E}[ \|\nabla f(w_{*}; \xi)\|^2 ]$; $\xi$ is a random variable, and $w_{*} = \arg \min_w F(w)$. 
\end{lem}

%By using \eqref{eq:stronglyconvex}, we then easily have the following corollary. 
%\begin{cor}\label{lem_bounded_secondmoment_03}
%Let Assumptions \ref{ass_stronglyconvex}, \ref{ass_convex} and \ref{ass_smooth} hold. Then, $\forall w \in \mathbb{R}^d$,
%\begin{gather*}
%\mathbb{E}\|\nabla f_i(w)\|^2 \leq  M \| \nabla F(w) \|^2 + N, 
%\end{gather*}
%where 
%\begin{align*}
%M = \frac{2 L}{\mu} \ \text{and} \ N = 2 \left( \frac{1}{n} \sum_{i=1}^{n} \|\nabla f_i(w_{*})\|^2 \right). \tagthis \label{eq_m_and_n}
%\end{align*}
%\end{cor}

%\begin{proof}
%By Lemma \ref{lem_bounded_secondmoment_04}, for $\forall w \in \mathbb{R}^d$, we have
%\begin{gather*}
%\mathbb{E}\|\nabla f_i(w)\|^2 \leq  4 L [ F(w) - F(w_{*}) ] + N \overset{\eqref{eq:stronglyconvex}}{\leq} \frac{2L}{\mu} \| \nabla F(w) \|^2 + N. 
%\end{gather*}
%\end{proof}

%Lemma \ref{lem_bounded_secondmoment_04} provides specific values of $M$ and $N$, which helps us not to make any assumption on the bound of $\mathbb{E}\|\nabla f_i(w)\|^2$. 

Using Lemma \ref{lem_bounded_secondmoment_04} and Super Martingale Convergence Theorem \cite{BertsekasSurvey} (Lemma \ref{prop_supermartingale} in the supplementary material), we can provide the sufficient condition for almost sure convergence of Algorithm \ref{sgd_algorithm} in the strongly convex case without assuming any bounded gradients. 

%\begin{thm}[Sufficient condition for almost sure convergence]\label{thm_general_02}
%Let Assumptions \ref{ass_stronglyconvex}, \ref{ass_convex} and \ref{ass_smooth} hold. Consider Algorithm \ref{sgd_algorithm} with a stepsize sequence such that
%\begin{align*}
%0 < \eta_t < \frac{\mu}{L^2} \ , \ \sum_{t=0}^{\infty} \eta_t = \infty \ \text{and} \ \sum_{t=0}^{\infty} \eta_t^2 < \infty. 
%\end{align*}
%Then, the following holds w.p.1 (almost surely)
%\begin{align*}
%F(w_{t}) - F(w_{*}) \to 0 \ \text{which implies} \ \| w_{t} - w_{*} \| \to 0. 
%\end{align*}
%
%%Moreover, there exists a compact set containing all iterations $w_{0},\dots,w_{t}$ w.p.1. 
%\end{thm}

\begin{thm}[Sufficient conditions for almost sure convergence]\label{thm_general_02_new_02}
Let Assumptions \ref{ass_stronglyconvex}, \ref{ass_smooth} and \ref{ass_convex} hold. Consider Algorithm \ref{sgd_algorithm} with a stepsize sequence such that
\begin{align*}
0 < \eta_t \leq \frac{1}{2 L} \ , \ \sum_{t=0}^{\infty} \eta_t = \infty \ \text{and} \ \sum_{t=0}^{\infty} \eta_t^2 < \infty. 
\end{align*}
Then, the following holds w.p.1 (almost surely)
\begin{align*}
\| w_{t} - w_{*} \|^2 \to 0. 
\end{align*}

%Moreover, there exists a compact set containing all iterations $w_{0},\dots,w_{t}$ w.p.1. 
\end{thm}

Note that the classical SGD proposed in \cite{RM1951} has learning rate satisfying conditions
\begin{align*}
\sum_{t=0}^{\infty} \eta_t = \infty \ \text{and} \ \sum_{t=0}^{\infty} \eta_t^2 < \infty
\end{align*}
However, the original analysis is performed under the bounded gradient assumption, as in \eqref{bounded_grad_ass}. 
In Theorem \ref{thm_general_02_new_02}, on the other hand, we do not use this assumption, but instead assume Lipschitz smoothness 
and convexity of the function realizations, which does not contradict the strong convexity of $F(w)$.

The following result establishes a sublinear convergence rate of SGD.   
%\begin{thm}\label{thm_res_sublinear_new_02}
%Let Assumptions \ref{ass_stronglyconvex}, \ref{ass_smooth} and \ref{ass_convex} hold. Consider Algorithm \ref{sgd_algorithm} with a step size sequence such that
%\begin{gather*}
%\eta_t = \frac{\alpha}{2 \alpha L + t} \leq \frac{1}{2 L} \ \text{for some } \alpha > \frac{1}{\mu}. 
%\end{gather*}
%Then, 
%\begin{gather}\label{sublinear_res_01_new_02}
%\mathbb{E}[\|w_{t} - w_{*}\|^2] \leq \frac{G}{2 \alpha L + t},
%\end{gather} 
%where $G = \max\{I,J\}$, and
%\begin{align*}
%I &= (2 \alpha L + 1) \left[ \left(1 - \frac{\mu}{2 L} \right)\| w_{0} - w_{*} \|^2 + \frac{N}{4 L^2}  \right]>0, \\
%J &= \frac{\alpha^2 N}{\alpha\mu - 1}>0, \\
%N &= 2 \mathbb{E}[ \|\nabla f(w_{*}; \xi_0)\|^2 ]. 
%\end{align*}
%%(Note that $I > 0$, $J > 0$.) 
%\end{thm}

\begin{thm}\label{thm_res_sublinear_new_02}
Let Assumptions \ref{ass_stronglyconvex}, \ref{ass_smooth} and \ref{ass_convex} hold. Let $E = \frac{2\alpha L}{\mu}$ with $\alpha=2$. Consider Algorithm \ref{sgd_algorithm} with a stepsize sequence such that $\eta_t = \frac{\alpha}{\mu(t+E)} \leq \eta_0=\frac{1}{2L}$. The expectation $\mathbb{E}[\|w_{t} - w_{*}\|^2]$ is at most
$$ \frac{4\alpha^2 N}{\mu^2} 
\frac{1 }{(t-T+E)} $$
for 
$$t\geq T =\frac{4L}{\mu}\max \{ \frac{L\mu}{N} \|w_{0} - w_{*}\|^2, 1\} - \frac{4L}{\mu}.$$ 
\end{thm}

\subsection{Convergence Analysis without Convexity}
%of SGD without Using Assumption~\ref{ass_convex}}

In this section, we provide the analysis of Algorithm \ref{sgd_algorithm} without using Assumption \ref{ass_convex}, that is, $f(w;\xi)$ is not necessarily convex. We  still do not need to impose the bounded stochastic gradient assumption, since we can derive an analogue of Lemma \ref{lem_bounded_secondmoment_04}, albeit with worse constant in the bound. 
\begin{lem}\label{lem_bounded_secondmoment_04_new}
Let Assumptions \ref{ass_stronglyconvex} and \ref{ass_smooth} hold. Then, for $\forall w \in \mathbb{R}^d$, 
\begin{gather}
\mathbb{E}[\|\nabla f(w; \xi)\|^2] \leq  4L \kappa [ F(w) - F(w_{*}) ] + N,
\label{afsfawfwaefwea_new} 
\end{gather}
where $\kappa = \frac{L}{\mu}$ and $N = 2 \mathbb{E}[ \|\nabla f(w_{*}; \xi)\|^2 ]$; $\xi$ is a random variable, and $w_{*} = \arg \min_w F(w)$.
\end{lem}

%By using \eqref{eq:stronglyconvex}, we then easily have the following corollary. 
%\begin{cor}\label{lem_bounded_secondmoment_03_new}
%Let Assumptions \ref{ass_stronglyconvex} and \ref{ass_smooth} hold. Then, $\forall w \in \mathbb{R}^d$,
%\begin{gather*}
%\mathbb{E}\|\nabla f_i(w)\|^2 \leq  M_2 \| \nabla F(w) \|^2 + N, 
%\end{gather*}
%where 
%\begin{align*}
%M_2 = \frac{2 L^2}{\mu^2} \ \text{and} \ N = 2 \left( \frac{1}{n} \sum_{i=1}^{n} \|\nabla f_i(w_{*})\|^2 \right). \tagthis \label{eq_m_and_n_new}
%\end{align*}
%\end{cor}

Based on the proofs of Theorems \ref{thm_general_02_new_02} and \ref{thm_res_sublinear_new_02}, we can easily have the following two results (Theorems \ref{thm_general_02_new_03} and  \ref{thm_res_sublinear_new_03}). 

%\begin{thm}[Sufficient condition for almost sure convergence]\label{thm_general_02_new}
%Let Assumptions \ref{ass_stronglyconvex} and \ref{ass_smooth} hold. Consider Algorithm \ref{sgd_algorithm} with a stepsize sequence such that
%\begin{align*}
%0 < \eta_t < \frac{\mu^2}{L^3} \ , \ \sum_{t=0}^{\infty} \eta_t = \infty \ \text{and} \ \sum_{t=0}^{\infty} \eta_t^2 < \infty. 
%\end{align*}
%Then, the following holds w.p.1 (almost surely)
%\begin{align*}
%F(w_{t}) - F(w_{*}) \to 0 \ \text{which implies} \ \| w_{t} - w_{*} \| \to 0. 
%\end{align*}
%
%%Moreover, there exists a compact set containing all iterations $w_{0},\dots,w_{t}$ w.p.1. 
%\end{thm}
%
%\begin{thm}\label{thm_res_sublinear_new}
%Let Assumptions \ref{ass_stronglyconvex} and \ref{ass_smooth} hold. Consider Algorithm \ref{sgd_algorithm} with a stepsize sequence such that
%\begin{gather*}
%\eta_t = \frac{\alpha}{\alpha L M_2 + t} \ \text{for some } \alpha > \frac{1}{\mu}. 
%\end{gather*}
%Then, 
%\begin{gather}\label{sublinear_res_01_new}
%\mathbb{E}[F(w_{t}) - F(w_{*})] \leq \frac{G}{\alpha L M_2 + t},
%\end{gather} 
%where $G = \max\{I,J\}$, and
%\begin{align*}
%I &= (\alpha L M_2 + 1) \Big[ \left(1 - \frac{\mu}{L M_2} \right)[F(w_{0}) - F(w_{*})] \\ & \qquad \qquad \qquad \qquad \qquad + \frac{N}{2 L M_2^2}  \Big], \\
%J &= \frac{\alpha^2 L N}{2(\alpha\mu - 1)}. 
%\end{align*}
%(Note that $I > 0$, $J > 0$.) And 
%\begin{align*}
%M_2 = \frac{2 L^2}{\mu^2} \ \text{and} \ N = 2 \left( \frac{1}{n} \sum_{i=1}^{n} \|\nabla f_i(w_{*})\|^2 \right).
%\end{align*}
%\end{thm}

\begin{thm}[Sufficient conditions for almost sure convergence]\label{thm_general_02_new_03}
 Let Assumptions \ref{ass_stronglyconvex} and \ref{ass_smooth} hold. Then, we can conclude the statement of Theorem \ref{thm_general_02_new_02} with the definition of the step size replaced by $0 < \eta_t \leq \frac{1}{2L \kappa}$ with $\kappa = \frac{L}{\mu}$.
% 
% Consider Algorithm~\ref{sgd_algorithm} with a stepsize sequence such that
%\begin{align*}
%0 < \eta_t \leq \frac{1}{2L \kappa} \ , \ \sum_{t=0}^{\infty} \eta_t = \infty \ \text{and} \ \sum_{t=0}^{\infty} \eta_t^2 < \infty,  
%\end{align*}
%where $\kappa = \frac{L}{\mu}$. Then, the following holds w.p.1 (almost surely)
%\begin{align*}
%\| w_{t} - w_{*} \|^2 \to 0. 
%\end{align*}
%%Moreover, there exists a compact set containing all iterations $w_{0},\dots,w_{t}$ w.p.1. 
\end{thm}

\begin{thm}\label{thm_res_sublinear_new_03}
Let Assumptions \ref{ass_stronglyconvex} and \ref{ass_smooth}  hold. 
Then, we can conclude the statement of Theorem \ref{thm_res_sublinear_new_02}  with the definition of the step size replaced by $\eta_t = \frac{\alpha}{\mu(t+E)} \leq \eta_0=\frac{1}{2L\kappa}$
 with $\kappa = \frac{L}{\mu}$ and $\alpha=2$, and all other occurrences of $L$ in $E$ and $T$ replaced by $L\kappa$.
%
%Let $E = \frac{2\alpha L\kappa}{\mu}$ with $\kappa = \frac{L}{\mu}$ and $\alpha=2$. Consider Algorithm \ref{sgd_algorithm} with a stepsize sequence such that $\eta_t = \frac{\alpha}{\mu(t+E)} \leq \eta_0=\frac{1}{2L\kappa}$. The expectation $\mathbb{E}[\|w_{t} - w_{*}\|^2]$ is at most
%$$ \frac{4\alpha^2 N}{\mu^2} 
%\frac{1 }{(t-T+E)} $$
%for $t\geq T =\frac{4L\kappa}{\mu}\max \{ \frac{L\kappa\mu}{N} \|w_{0} - w_{*}\|^2, 1\} - \frac{4L\kappa}{\mu}$. 
\end{thm}

We compare our result in Theorem \ref{thm_res_sublinear_new_03} with that in \cite{bottou2016optimization} in the following remark. 
\begin{rem}
By strong convexity of $F$, Lemma \ref{lem_bounded_secondmoment_04_new} implies $\mathbb{E}[\|\nabla f(w; \xi)\|^2] \leq  2 \kappa^2 \| \nabla F(w) \|^2 + N$, for $\forall w \in \mathbb{R}^d$, where $\kappa = \frac{L}{\mu}$ and $N = 2 \mathbb{E}[ \|\nabla f(w_{*}; \xi)\|^2 ]$. We can now
substitute  the value $M = 2 \kappa^2$ into Theorem 4.7 in \cite{bottou2016optimization}. We  observe that the resulting initial learning rate in \cite{bottou2016optimization} has to satisfy $\eta_0 \leq \frac{1}{2 L \kappa^2}$ while our results allows $\eta_0 = \frac{1}{2 L \kappa}$. 
We are able to achieve this improvement by introducing Assumption  \ref{ass_smooth}, which holds for many ML problems. 

Recall that under  Assumption \ref{ass_convex}, our initial learning rate is $\eta_0 = \frac{1}{2 L}$ (in Theorem \ref{thm_res_sublinear_new_02}). Thus  Assumption \ref{ass_convex} provides further  improvement of the conditions on the learning rate. 
\end{rem}

\section{Asynchronous Stochastic Optimization aka Hogwild!}
\label{general}

Hogwild! \cite{Hogwild} is an asynchronous stochastic optimization method where writes to and reads from vector  positions in shared memory can be inconsistent (this corresponds to (\ref{eqwM2}) as we shall see). 
However, as mentioned in~\cite{ManiaPanPapailiopoulosEtAl2015}, for the purpose of analysis  the method in \cite{Hogwild} performs single vector entry updates that are randomly selected from the non-zero entries of the computed gradient as in (\ref{eqwM1}) (explained later) and requires the assumption of consistent vector reads together with the bounded gradient assumption to prove convergence. Both \cite{ManiaPanPapailiopoulosEtAl2015} and \cite{DeSaZhangOlukotunEtAl2015} prove the same result for fixed step size based on the assumption of bounded stochastic gradients in the strongly convex case but now without assuming consistent vector reads and writes.
In these works the fixed step size $\eta$ must depend on $\sigma$ from the bounded gradient assumption, however, one does not usually know $\sigma$ and thus, we cannot compute a suitable $\eta$ a-priori. 

As claimed by the authors in \cite{ManiaPanPapailiopoulosEtAl2015}, they can eliminate the bounded gradient assumption in their analysis of Hogwild!, which however was only mentioned as a remark without proof. On the other hand, the authors of recent unpublished work~\cite{Leblond2018} formulate and prove, without  the bounded gradient assumption, a precise theorem about the convergence rate of Hogwild! of the form
$$ \mathbb{E}[\|w_{t} - w_* \|^2] \leq (1-\rho)^t(2 \|w_0-w_*\|^2) + b,$$
where $\rho$ is a function of several parameters but independent of the fixed chosen step size $\eta$  and where $b$ is a function of several parameters and has a linear dependency with respect to the fixed step size, i.e., $b=O(\eta)$.

In this section, we discuss  the convergence of Hogwild!  with \textbf{diminishing} stepsize where writes to and reads from vector  positions in shared memory can be  \textbf{inconsistent}. This is a slight modification of the original Hogwild! where the stepsize is fixed. 
In our analysis we also  \textbf{do not use the bounded gradient assumption} as in \cite{Leblond2018}. Moreover, (a) we focus on solving the   \textbf{more general problem} in \eqref{main_prob_expected_risk}, while \cite{Leblond2018} considers the specific case of  the ``finite-sum'' 
	problem in \eqref{main_prob}, and (b) we show that our analysis generalizes to the  \textbf{non-convex case}, i.e., we do not need to assume functions $f(w;\xi)$ are convex (we only require $F(w) =  \mathbb{E}[f(w;\xi)]$ to be strongly convex) as opposed to the assumption in \cite{Leblond2018}.

\subsection{Recursion}

We first formulate a general recursion for $w_t$ to which our analysis applies, next we will explain how the different variables in the recursion interact and describe two special cases, and finally we present pseudo code of the algorithm using the recursion.  

The recursion explains which positions in $w_t$ should be updated in order to compute $w_{t+1}$. Since $w_t$ is stored in shared memory and is being updated in a possibly non-consistent way by multiple cores who each perform recursions, the shared memory will contain a vector $w$ whose entries represent a mix of updates. That is, before performing the computation of a recursion, a core will first read  $w$ from shared memory, however, while reading $w$ from shared memory, the entries in $w$ are being updated out of order. The final vector $\hat{w}_t$ read by the core represents an aggregate of a mix of updates in previous iterations.

The general recursion is defined as follows: For $t\geq 0$,
\begin{equation}
 w_{t+1} = w_t - \eta_t d_{\xi_t}  S^{\xi_t}_{u_t} \nabla f(\hat{w}_t;\xi_t),\label{eqwM}
 \end{equation}
 where
 \begin{itemize}
 \item $\hat{w}_t$ represents the vector used in computing the gradient $\nabla f(\hat{w}_t;\xi_t)$ and whose entries have been read (one by one)  from  an aggregate of a mix of  previous updates that led to $w_{j}$, $j\leq t$, and
 \item the $S^{\xi_t}_{u_t}$ are diagonal 0/1-matrices with the property that there exist real numbers $d_\xi$ satisfying
\begin{equation} d_\xi \mathbb{E}[S^\xi_u | \xi] = D_\xi, \label{eq:SexpM} \end{equation}
where the expectation is taken over $u$ and $D_\xi$ is the diagonal 0/1 matrix whose $1$-entries correspond to the non-zero positions in $\nabla f(w;\xi)$, i.e., the $i$-th entry of $D_\xi$'s diagonal is equal to 1 if and only if there exists a $w$ such that the $i$-th position of $\nabla f(w;\xi)$ is non-zero. 
\end{itemize}

The role of matrix $S^{\xi_t}_{u_t}$ is that it filters which positions of gradient $\nabla f(\hat{w}_t;\xi_t)$ play a role in (\ref{eqwM}) and need to be computed. Notice that $D_\xi$ represents the support of $\nabla f(w;\xi)$; by $|D_\xi|$ we denote the number of 1s in $D_\xi$, i.e., $|D_\xi|$ equals the size of the support of $\nabla f(w;\xi)$.

We will restrict ourselves to choosing (i.e., fixing a-priori) {\em non-empty} matrices  $S^\xi_u$ that ``partition'' $D_\xi$ in $D$ approximately ``equally sized'' $S^\xi_u$: 
$$ \sum_u S^\xi_u = D_\xi, $$
where each matrix $S^\xi_u$ has either $\lfloor |D_\xi|/D \rfloor$ or $\lceil |D_\xi|/D \rceil$ ones on its diagonal. We uniformly choose one of the matrices $S^{\xi_t}_{u_t}$ in (\ref{eqwM}), hence, $d_\xi$ equals the number of matrices $S^\xi_u$, see (\ref{eq:SexpM}).

In other to explain recursion (\ref{eqwM}) we first consider two special cases. For $D=\bar{\Delta}$, where 
$$ \bar{\Delta} = \max_\xi \{ |D_\xi|\}$$
represents the maximum number of non-zero positions in any gradient computation $f(w;\xi)$, we have that for all $\xi$, there are exactly $|D_\xi|$ diagonal matrices $S^\xi_u$ with a single 1 representing each of the elements in $D_\xi$. Since  $p_\xi(u)= 1/|D_\xi|$ is the uniform distribution, we have $\mathbb{E}[S^\xi_u | \xi] = D_\xi / |D_\xi|$, hence, $d_\xi = |D_\xi|$. This gives the recursion
\begin{equation}
 w_{t+1} = w_t - \eta_t |D_\xi|  [ \nabla f(\hat{w}_t;\xi_t)]_{u_t},\label{eqwM1}
 \end{equation}
 where $ [ \nabla f(\hat{w}_t;\xi_t)]_{u_t}$ denotes the $u_t$-th position of $\nabla f(\hat{w}_t;\xi_t)$ and where $u_t$ is a uniformly selected position that corresponds to a non-zero entry in  $\nabla f(\hat{w}_t;\xi_t)$.
 
At the other extreme, for $D=1$, we have exactly one matrix $S^\xi_1=D_\xi$ for each $\xi$, and we have $d_\xi=1$. This gives the recursion
\begin{equation}
 w_{t+1} = w_t - \eta_t  \nabla f(\hat{w}_t;\xi_t).\label{eqwM2}
 \end{equation}
Recursion (\ref{eqwM2}) represents Hogwild!. In a single-core setting where updates are done in a consistent way and $\hat{w}_t=w_t$ yields SGD.

%. Experiments in Section \ref{sec_experiments} indicate that while the expected convergence rates do seem to look similar, their variances differ.
 
 Algorithm \ref{HogWildAlgorithm} gives the pseudo code corresponding to recursion (\ref{eqwM}) with our choice of sets $S^\xi_u$ (for parameter $D$).
 
 \begin{algorithm}
\caption{Hogwild! general recursion}
\label{HogWildAlgorithm}
\begin{algorithmic}[1]

   \STATE {\bf Input:} $w_{0} \in \R^d$
   \FOR{$t=0,1,2,\dotsc$ {\bf in parallel}} 
    
  \STATE read each position of shared memory $w$
  denoted by $\hat w_t$  {\bf (each position read is atomic)}
  \STATE draw a random sample $\xi_t$ and a random ``filter'' $S^{\xi_t}_{u_t}$
  \FOR{positions $h$ where $S^{\xi_t}_{u_t}$ has a 1 on its diagonal}
   \STATE compute $g_h$ as the gradient $\nabla f(\hat{w}_t;\xi_t)$ at position $h$
   \STATE add $\eta_t d_{\xi_t} g_h$ to the entry at position $h$ of $w$ in shared memory {\bf (each position update is atomic)}
   \ENDFOR
   \ENDFOR
\end{algorithmic}
\end{algorithm}

\subsection{Analysis}

Besides Assumptions \ref{ass_stronglyconvex}, \ref{ass_smooth}, and for now \ref{ass_convex}, we assume the following assumption regarding a parameter $\tau$, called the delay, which indicates which updates in previous iterations have certainly made their way into shared memory $w$.

\begin{ass}[Consistent with delay $\tau$]
\label{ass_tau}
We say that shared memory is consistent with delay $\tau$  with respect to recursion (\ref{eqwM}) if, for all $t$, vector $\hat{w}_t$ includes the aggregate of the updates up to and including those made during the $(t-\tau)$-th iteration (where (\ref{eqwM}) defines the $(t+1)$-st iteration). Each position read from shared memory is atomic and each position update to shared memory is atomic (in that these cannot be interrupted by another update to the same position).
\end{ass}

In other words in the $(t+1)$-th iteration,  $\hat{w}_t$ equals  $w_{t-\tau}$ plus some subset of position updates made during iterations $t-\tau, t-\tau+1, \ldots, t-1$. We assume that there exists a constant delay $\tau$ satisfying Assumption \ref{ass_tau}.

%Appendix \ref{sec:Hogwild_insconsistent_read_write} 
The supplementary material proves the following theorem where 
$$\bar{\Delta}_D \eqdef D \cdot \mathbb{E}[\lceil |D_\xi|/D \rceil].$$

\begin{thm}
\label{theorem:Hogwild_newnew1}
Suppose Assumptions \ref{ass_stronglyconvex}, \ref{ass_smooth}, \ref{ass_convex} and \ref{ass_tau}  and  consider Algorithm~\ref{HogWildAlgorithm} for sets $S^\xi_u$ with parameter $D$. Let  $\eta_t = \frac{\alpha_t}{\mu(t+E)}$ with $4\leq \alpha_t \leq\alpha$ and $E = \max\{ 2\tau, \frac{4 L \alpha D}{\mu}\}$. Then, the expected number of single vector entry updates after $t$ iterations is equal to
$$t' = t \bar{\Delta}_D /D$$  and
 expectations
 $\mathbb{E}[\|\hat{w}_{t} - w_* \|^2]$ and $\mathbb{E}[\|w_{t} - w_* \|^2]$ are at most
  \begin{eqnarray*}
  \frac{4\alpha^2D N}{\mu^2} \frac{t}{(t + E - 1)^2} + O\left(\frac{\ln t}{(t+E-1)^{2}}\right).
%\frac{4\alpha^2 \bar{\Delta}_D N}{\mu^2} \frac{1}{t'} + O\left(\frac{\ln t'}{t'^{2}}\right).
\end{eqnarray*} 
\end{thm}

In terms of $t'$, the expected number  single vector entry updates after $t$ iterations, $\mathbb{E}[\|\hat{w}_{t} - w_* \|^2]$ and $\mathbb{E}[\|w_{t} - w_* \|^2]$ are at most
$$\frac{4\alpha^2 \bar{\Delta}_D N}{\mu^2} \frac{1}{t'} + O\left(\frac{\ln t'}{t'^{2}}\right).$$

\begin{rem}
In (\ref{eqwM1}) $D=\bar{\Delta}$, hence, $\lceil |D_\xi|/D \rceil =1$ and $\bar{\Delta}_D = \bar{\Delta} = \max_\xi \{|D_\xi|\}$. In (\ref{eqwM2}) $D=1$, hence, $\bar{\Delta}_D = \mathbb{E}[|D_\xi|]$. This shows that the upper bound in Theorem \ref{theorem:Hogwild_newnew1} is better for (\ref{eqwM2}) with $D=1$. If we assume no delay, i.e. $\tau=0$, in addition to $D=1$, then we obtain SGD. Theorem \ref{thm_res_sublinear_new_02} shows that, measured in $t'$, we obtain the upper bound
$$\frac{4\alpha_{SGD}^2 \bar{\Delta}_D N}{\mu^2} \frac{1}{t'} $$
with $\alpha_{SGD}=2$ as opposed to $\alpha\geq 4$.

With respect to parallelism, SGD assumes a single core, while (\ref{eqwM2}) and (\ref{eqwM1}) allow multiple cores. 
Notice that recursion (\ref{eqwM1}) allows us to partition the position of the shared memory among the different processor cores in such a way that each partition can only be updated by its assigned core and where partitions can be read by all cores. This allows optimal resource sharing and could make up for the difference between $\bar{\Delta}_D$ for (\ref{eqwM1}) and (\ref{eqwM2}). We hypothesize that, for a parallel implementation, $D$ equal to a fraction of $\bar{\Delta}$ will lead to best performance.
% with the caveat that, even though recursions (\ref{eqwM1}) and (\ref{eqwM2}) lead to similar upper bounds on the expected convergence rate in our theoretical analysis, in practice the two recursions may show significant differences in the observed variances of the convergence rates, see Section \ref{sec_experiments}.
\end{rem}

\begin{rem}
Surprisingly, the leading term of the upper bound on the convergence rate is independent of delay $\tau$. On one hand, one would expect that a more recent read which contains more of the updates done during the last $\tau$ iterations will lead to better convergence. When inspecting the second order term in the proof in 
%Appendix \ref{sec:Hogwild_insconsistent_read_write} 
the supplementary material, we do see that a smaller $\tau$ (and/or smaller sparsity) makes the convergence rate smaller. That is, asymptotically $t$ should be large enough as a function of $\tau$ (and other parameters) in order for the leading term to dominate. 

Nevertheless, in asymptotic terms (for larger $t$) the dependence on $\tau$ is not noticeable.  In fact, 
%Appendix \ref{sec:Hogwild_insconsistent_read_write} 
the supplementary material shows that we may allow $\tau$ to be a monotonic increasing function of $t$ with
$$\frac{2 L \alpha D}{\mu}\leq \tau(t)\leq \sqrt{t \cdot L(t)},$$
where $L(t)=\frac{1}{\ln t} - \frac{1}{(\ln t)^2}$ (this will make $E = \max\{ 2\tau(t), \frac{4 L \alpha D}{\mu}\}$ also a function of $t$). The leading term of the convergence rate does not change while the second order terms increase to $O(\frac{1}{t\ln t})$. We show that, for
$$ t\geq T_0 =  \exp[ 2\sqrt{\Delta}(1+\frac{(L+\mu)\alpha}{\mu})],$$
where $\Delta= \max_i \Prob \left(   i \in  D_\xi  \right)$ measures sparsity,
the higher order terms that contain $\tau(t)$ (as defined above) are at most the leading term.

Our intuition behind this phenomenon is that for large $\tau$, all the last $\tau$ iterations before the $t$-th iteration use vectors $\hat{w}_j$  with entries that are dominated by the aggregate of updates that happened till iteration $t-\tau$. Since the average sum of the updates during the last $\tau$ iterations is equal to 
 \begin{equation} - \frac{1}{\tau} \sum_{j=t-\tau}^{t-1} \eta_j d_{\xi_j}  S^{\xi_j}_{u_j} \nabla f(\hat{w}_j;\xi_t) \label{Eqtau} \end{equation}
 and all $\hat{w}_j$ look alike in that they mainly represent learned information before the $(t-\tau)$-th iteration, (\ref{Eqtau}) becomes an estimate 
 of the expectation of  (\ref{Eqtau}), i.e.,
 %\begin{eqnarray}  && 
 \begin{equation}
 \sum_{j=t-\tau}^{t-1} \frac{-\eta_j}{\tau}  \mathbb{E}[d_{\xi_j}  S^{\xi_j}_{u_j} \nabla f(\hat{w}_j;\xi_t)] 
 %\nonumber \\ &=& 
 =\sum_{j=t-\tau}^{t-1} \frac{-\eta_j}{\tau}  \nabla F(\hat{w}_j). \label{EGD} 
 \end{equation}
 %\end{eqnarray}
 This looks like GD which in the strong convex case has convergence rate $\leq c^{-t}$ for some constant $c>1$. This already shows that larger $\tau$ could help convergence as well. However, 
 %for problem (\ref{main_prob_expected_risk}) 
 estimate (\ref{Eqtau}) has estimation noise with respect to (\ref{EGD}) which explains why in this thought experiment we cannot attain $c^{-t}$ but can only reach a much smaller convergence rate of e.g. $O(1/t)$ as in Theorem \ref{theorem:Hogwild_newnew1}. 
 
 Experiments in Section \ref{sec_experiments} confirm our analysis.
\end{rem}
 
\begin{rem} 
The higher order terms in the proof 
%in Appendix \ref{sec:Hogwild_insconsistent_read_write} 
in the supplementary material show that, as in Theorem \ref{thm_res_sublinear_new_02},  the expected convergence rate in Theorem \ref{theorem:Hogwild_newnew1} depends on $\|w_0-w_*\|^2$. The proof shows that, for 
$$ t \geq T_1 = \frac{\mu^2}{\alpha^2 N D}\|w_0-w_*\|^2,$$
the higher order term that contains $\|w_0-w_*\|^2$ is at most the leading term. This is comparable to $T$ in Theorem \ref{thm_res_sublinear_new_02} for SGD.
\end{rem}

\begin{rem}
Step size $\eta_t=\frac{\alpha_t}{\mu(t+E)}$ with $4\leq \alpha_t \leq\alpha$ can be chosen to be fixed during periods whose ranges exponentially increase. For $t+E\in [2^h,2^{h+1})$ we define $\alpha_t= \frac{4(t+E)}{2^h}$. Notice that $4\leq \alpha_t<8$ which satisfies the conditions of Theorem \ref{theorem:Hogwild_newnew1} for $\alpha=8$. This means that we can choose 
$$\eta_t = \frac{\alpha_t}{\mu(t+E)}=\frac{4}{\mu 2^h}$$
as step size for $t+E\in [2^h,2^{h+1})$. This choice for $\eta_t$ allows changes in $\eta_t$ to be easily synchronized between cores since these changes only happen when $t+E=2^h$ for some integer $h$. That is, if each core is processing iterations at the same speed, then each core on its own may reliably assume that after having processed $(2^h-E)/P$ iterations the aggregate of all $P$ cores has approximately processed $2^h-E$ iterations. So, after $(2^h-E)/P$ iterations a core will increment its version of $h$ to $h+1$. This will introduce some noise as the different cores will not increment their $h$ versions at exactly the same time, but this only happens during a small interval around every $t+E=2^h$. This will occur rarely for larger $h$.
%Since this happens at exponentially in time separated moments, this will occur rarely.
\end{rem}

\subsection{Convergence Analysis without Convexity}
 
 % In Appendix \ref{sec:Hogwild_insconsistent_read_write} 
 In the supplementary material, we also show that the proof of Theorem \ref{theorem:Hogwild_newnew1} can easily be modified such that Theorem \ref{theorem:Hogwild_newnew1} with $E\geq \frac{4L\kappa \alpha D}{\mu}$ also holds in the non-convex case of the component functions, i.e., we do not need  Assumption \ref{ass_convex}. Note that this case is not analyzed in \cite{Leblond2018}.

\begin{thm}\label{thm_6}
Let Assumptions \ref{ass_stronglyconvex} and \ref{ass_smooth}  hold. 
Then, we can conclude the statement of Theorem \ref{theorem:Hogwild_newnew1} with $E\geq \frac{4L\kappa \alpha D}{\mu}$ for $\kappa = \frac{L}{\mu}$.
\end{thm}

%\input{Fixed_stepsize.tex}

%========================

%\clearpage 

%\input{hogwild_description_UCONN}

%========================

\section{Numerical Experiments}\label{sec_experiments}

For our numerical experiments, we consider the finite sum minimization problem in \eqref{main_prob}. 
%\begin{equation*}
%\min_{w \in \mathbb{R}^d} \left\{ F(w) = \frac{1}{n} \sum_{i=1}^n f_i(w) \right\}.
% \end{equation*}
We consider $\ell_2$-regularized logistic regression problems with
\begin{align*}
f_i(w) = \log(1 + \exp(-y_i \langle x_i, w\rangle )) +
\frac{\lambda}{2} \| w \|^2,
\end{align*}
where the penalty parameter $\lambda$ is set to $1/n$, a widely-used
value in  literature \cite{SAG}.

\begin{figure}[h]
 \centering
 \includegraphics[width=0.45\textwidth]{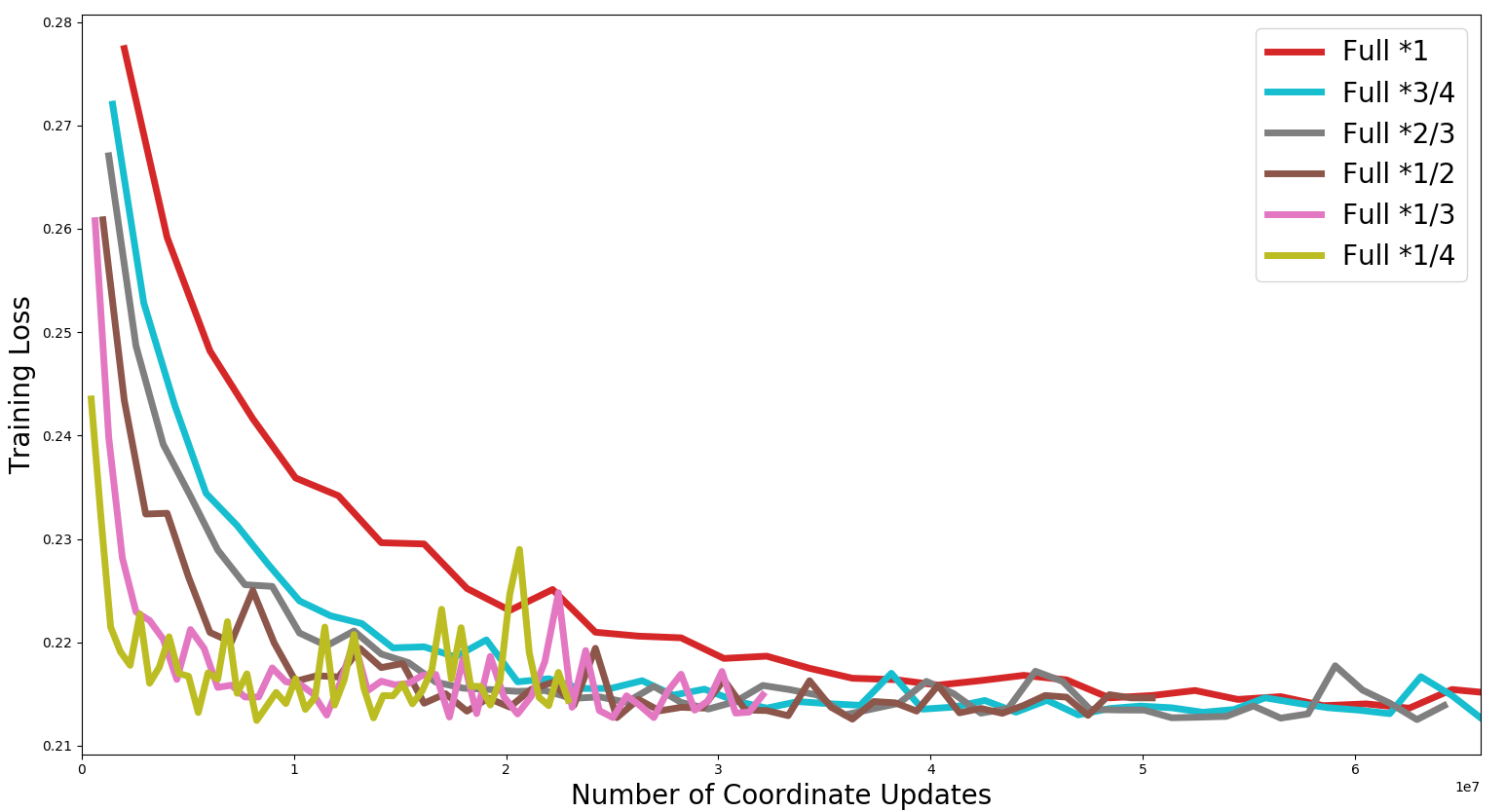}
 \includegraphics[width=0.45\textwidth]{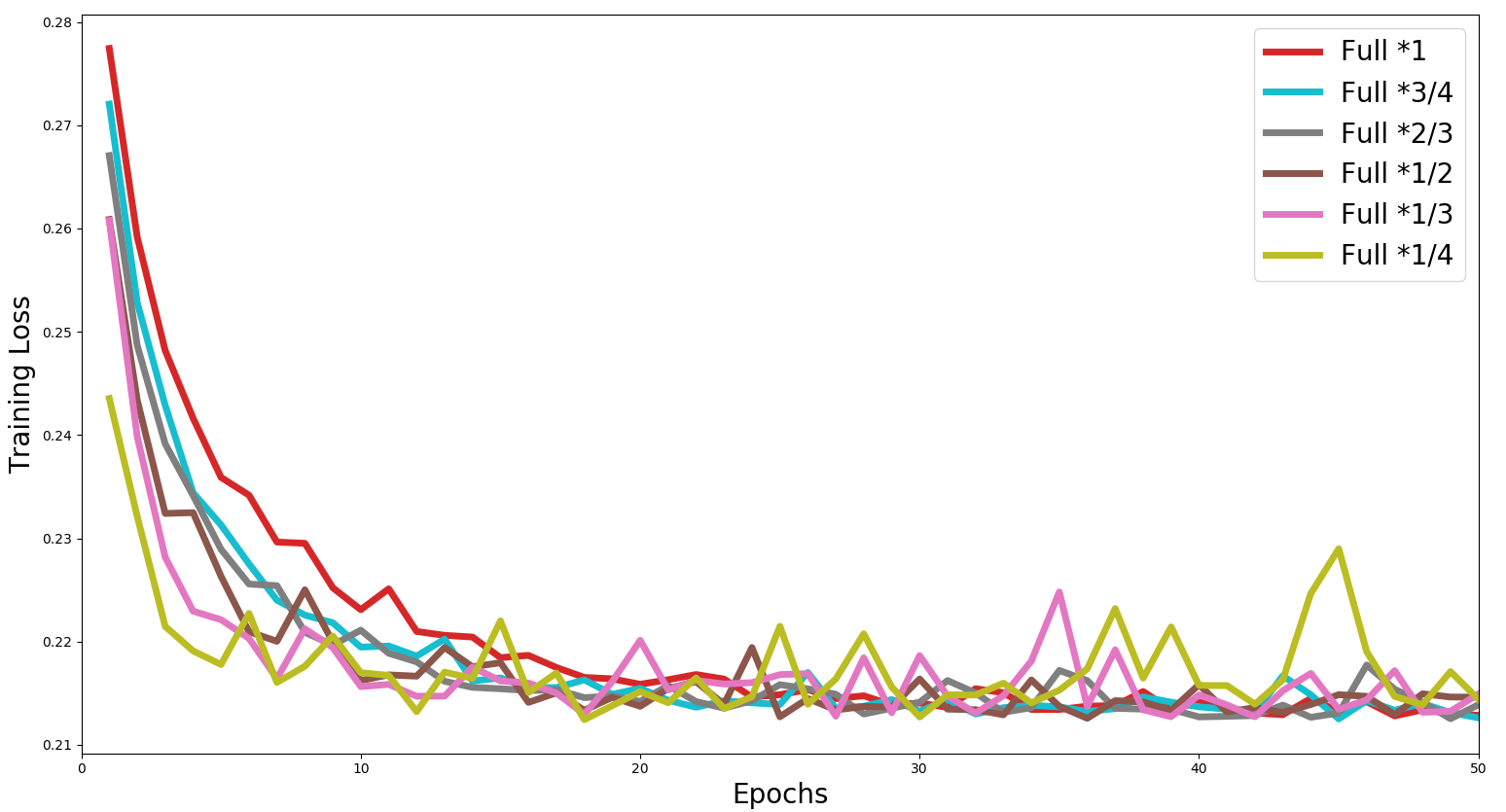}
   \caption{\textit{ijcnn1} for different fraction of non-zero set}
  \label{figure_01_a}
 \end{figure}
 
 \begin{figure}[h]
 \centering
 \includegraphics[width=0.45\textwidth]{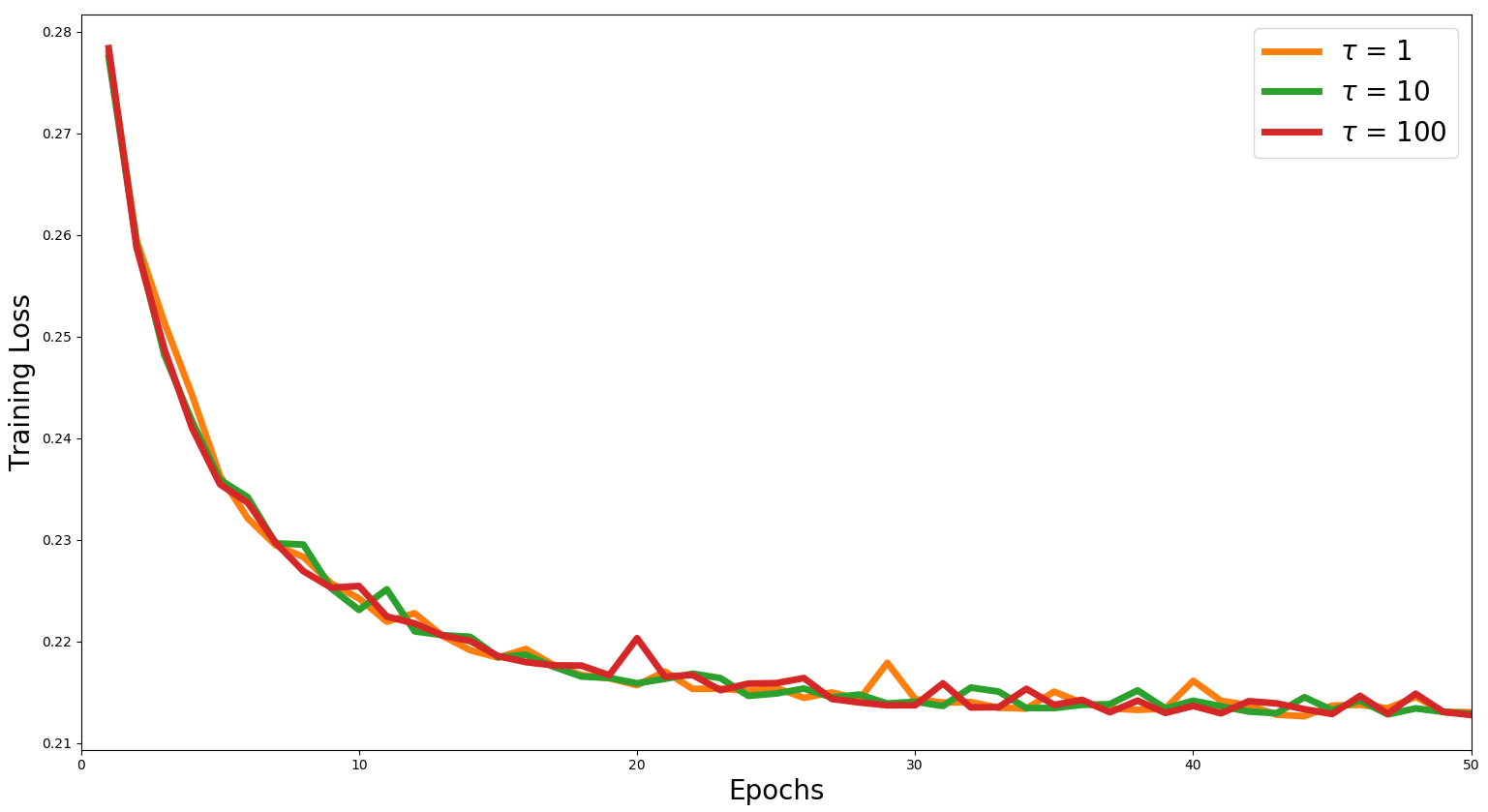} 
   \caption{\textit{ijcnn1} for different $\tau$ with the whole non-zero set}
  \label{figure_01_b}
 \end{figure}
 
We conducted experiments on a single core for Algorithm \ref{HogWildAlgorithm} on two popular datasets \texttt{ijcnn1} ($n = 91, 701$ training data) and \texttt{covtype} ($n = 406,709$ training data) from the LIBSVM\footnote{http://www.csie.ntu.edu.tw/$\sim$cjlin/libsvmtools/datasets/} website. Since we are interested in the expected convergence rate with respect to the number of iterations, respectively number of single position vector updates, we do not need a parallelized multi-core simulation to confirm our analysis. The impact of efficient resource scheduling over multiple cores leads to a  performance improvement complementary to our analysis of (\ref{eqwM}) (which, as discussed, lends itself for an efficient parallelized implementation). We experimented with 10 runs and reported the average results. We choose the step size based on Theorem \ref{theorem:Hogwild_newnew1}, i.e, $\eta_t = \frac{4}{\mu(t+E)}$ and $E = \max\{ 2\tau, \frac{16 L D}{\mu}\}$. For each fraction $v\in \{1,3/4,2/3,1/2,1/3, 1/4\}$ we performed the following experiment: In Algorithm \ref{HogWildAlgorithm}  we  choose each ``filter'' matrix $S^{\xi_t}_{u_t}$ to correspond with a random subset of size $v|D_{\xi_t}|$ of the non-zero positions of $D_{\xi_t}$ (i.e., the support of the gradient corresponding to $\xi_t$). In addition we use $\tau=10$. For the two datasets, Figures \ref{figure_01_a} and \ref{figure_02_a} plot the training loss for each fraction with $\tau=10$. The top plots have $t'$, the number of coordinate updates, for the horizontal axis. The bottom plots have the number of epochs, each epoch counting $n$ iterations, for the horizontal axis. The results show that each fraction shows a sublinear expected convergence rate of $O(1/t')$; the smaller fractions exhibit larger deviations but do seem to converge faster to the minimum solution.

%We plotted different fractions of non-zero set with $\tau = 10$ and compare their train loss based on number of coordinate updates and number of epochs as shown in Figures \ref{figure_01_a} and \ref{figure_02_a}. 

In Figures \ref{figure_01_b} and \ref{figure_02_b}, we show experiments with different values of $\tau \in \{1, 10, 100\}$ where we use the whole non-zero set of gradient positions (i.e., $v=1$) for the update.  Our analysis states that, for $t= 50$ epochs times $n$ iterations per epoch, $\tau$ can be as large as $\sqrt{t\cdot L(t)}=524$ for \texttt{ijcnn1} and $1058$ for \texttt{covtype}. The experiments indeed show that $\tau\leq 100$ has little effect on the expected convergence rate.

\begin{figure}[h]
 \centering
 \includegraphics[width=0.45\textwidth]{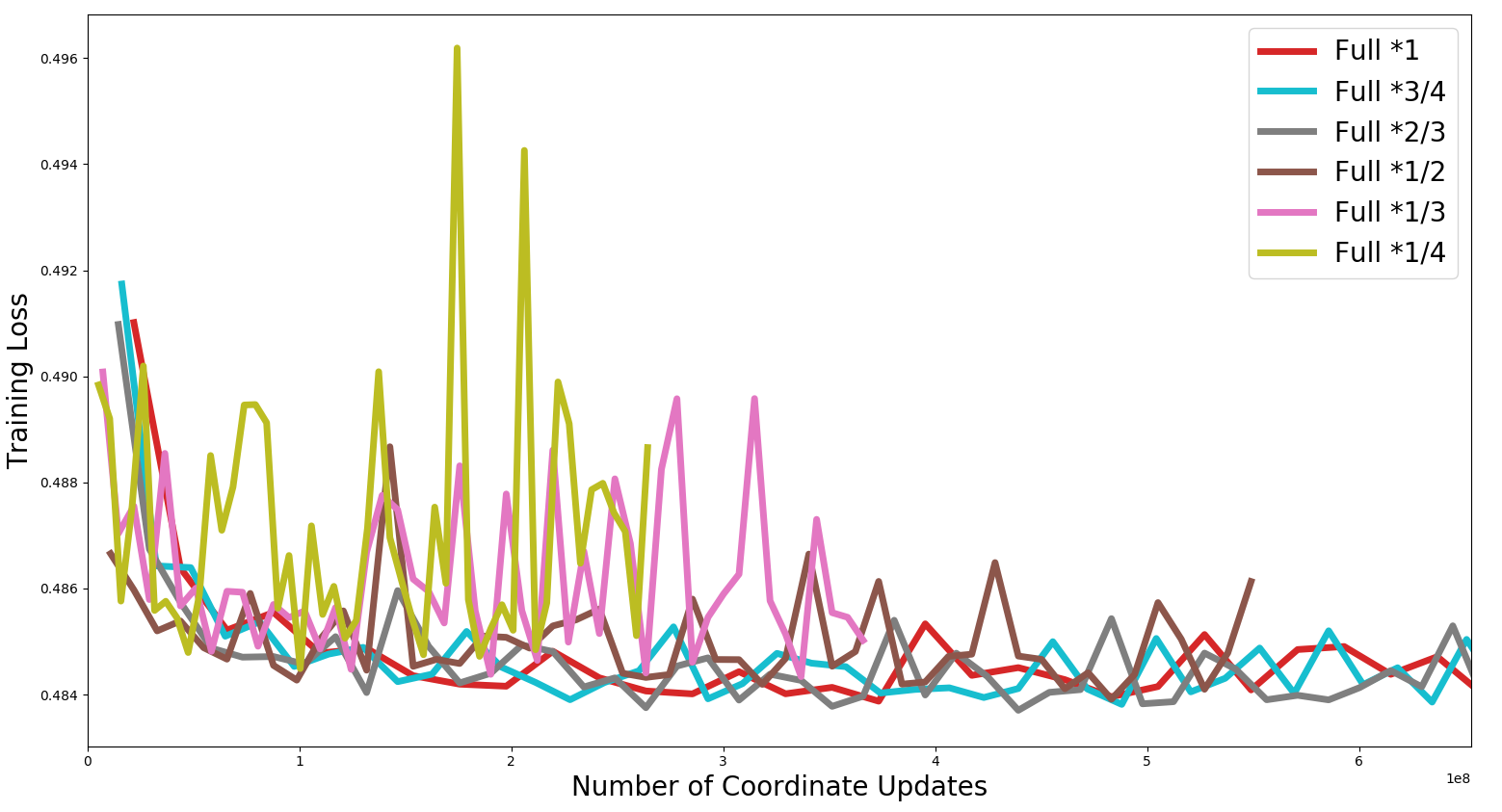}
 \includegraphics[width=0.45\textwidth]{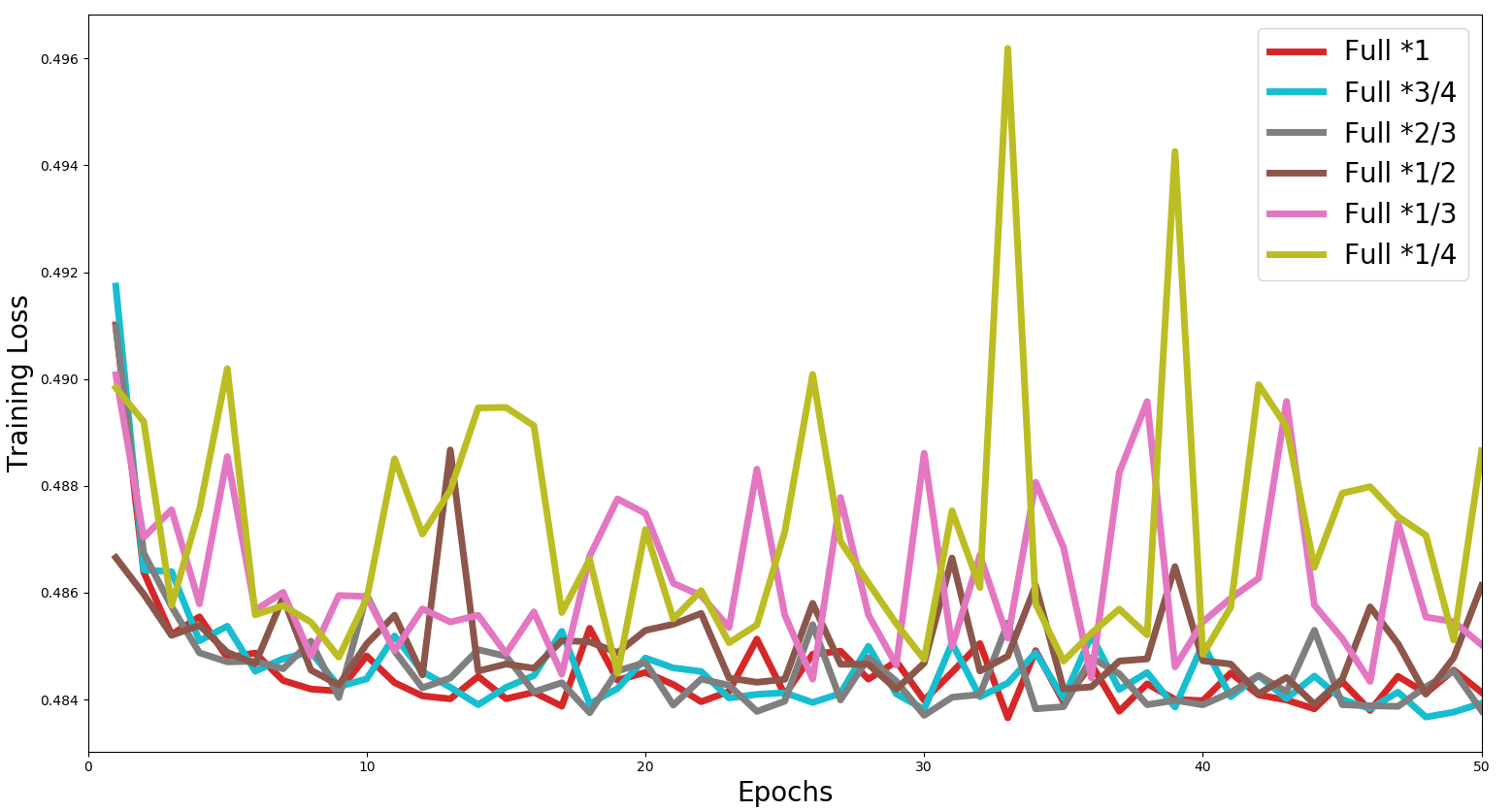}
   \caption{\textit{covtype} for different fraction of non-zero set}
  \label{figure_02_a}
 \end{figure}
 
 \begin{figure}[h]
 \centering
 \includegraphics[width=0.45\textwidth]{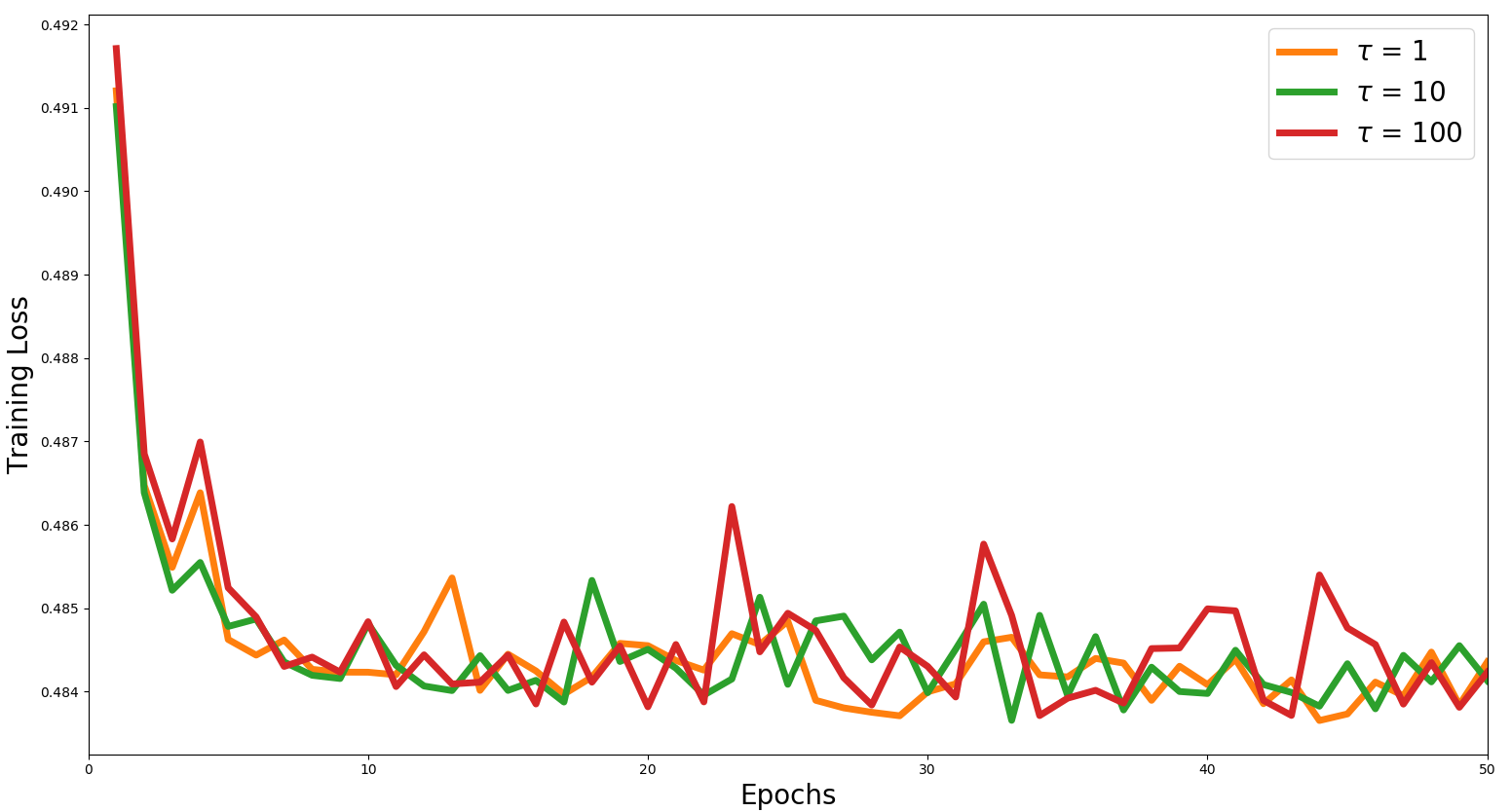} 
   \caption{\textit{covtype} for different $\tau$ with the whole non-zero set}
  \label{figure_02_b}
 \end{figure}

%========================

\section{Conclusion}

We have provided the analysis of stochastic gradient algorithms with diminishing step size in the strongly convex case under
the condition of Lipschitz continuity of the individual function realizations, but without requiring any bounds on
the stochastic gradients. We showed almost sure convergence of SGD and provided sublinear upper bounds for the expected convergence rate of a general recursion which includes Hogwild! for inconsistent reads and writes as a special case. We also provided new intuition which will help understanding convergence as observed in practice.

\section*{Acknowledgement}

The authors would like to thank the reviewers for useful suggestions which helped to improve the exposition in the paper. The authors also would like to thank Francesco Orabona for his valuable comments and suggestions. 

Lam M. Nguyen was partially supported by NSF Grants CCF 16-18717. Phuong Ha Nguyen and Marten van Dijk were supported in part by AFOSR MURI under award number FA9550-14-1-0351. Katya Scheinberg was partially supported by NSF Grants CCF 16-18717 and CCF 17-40796. Martin Tak\'{a}\v{c} was supported by  U.S. National Science Foundation, under award number NSF:CCF:1618717, NSF:CMMI:1663256 and NSF:CCF:1740796.

\bibliography{references}
\bibliographystyle{icml2018}

% \end{document}

\clearpage
\onecolumn
\appendix

\icmltitle{SGD and Hogwild! Convergence Without the Bounded Gradients Assumption \\
			Supplementary Material, ICML 2018}

%\section{Notations}
%
%For the sake of analysis in Section \ref{sec:ASO}, we introduce notations used in Table~\ref{table:1}. 
%
%\begin{table}[h!]
%\centering
%\begin{tabular}{|c c c|} 
% \hline
% Symbols & Definition & Where\\ [0.5ex] 
% \hline\hline
% $N$  & $N=\frac{2}{n}\sum_{i=1}^n\|\nabla f_i (w_*) \|$& Corollary~\ref{lem_bounded_secondmoment_03} \\
% $M$  & $M=\frac{2L}{\mu}$ & Corollary~\ref{lem_bounded_secondmoment_03} \\
% $\mathcal{F}_t$& $\mathcal{F}_t=\sigma(w_1,i_1,\rho_1,\dotsc,i_{t-1},\rho_{t-1})$ & Lemma~\ref{lemma:hogwild_1}  \\
% $a_t$ & $a_t= L^2\eta_t - \frac{L^3M}{2}\eta^2_t$ & Lemma~\ref{lemma:hogwild_1} \\
% $b_t$ & $b_t= \frac{3}{4}\eta_t - \frac{LM}{4}\eta^2_t$ & Lemma~\ref{lemma:hogwild_1}\\  
% $N_{\infty}$ & $N_{\infty} = \max_{1 \leq i \leq n} \| \nabla f_i(w_{*})\|$ & Lemma~\ref{lemma:hogwild_2} \\ 
% $m_{t-1}$ & $m_{t-1}=\sum_{j=1}^{t-1}\| \nabla f_{i_j}(w_{j-\rho_j})-\nabla f_{i_j}(w_{*}) \|^2$&Lemma~\ref{lemma:hogwild_2}  \\
% $E$  & $E \geq \max\{2, 2R, \frac{LM\alpha^2}{2\mu}\}$ and $\frac{E-1}{1+\ln E} \geq \frac{\alpha^316R^2L^3T}{\mu^3 e}$& Lemmas~\ref{lemma:hogwild_3},~\ref{lemma:hogwild_4},~\ref{lemma:hogwild_4b},~\ref{lemma:hogwild_04c}\\
% $\eta_t$  & $\eta_t = \frac{\alpha_t}{\mu (t+E)}, 2 \leq \alpha_t \leq \alpha$ &Lemmas~\ref{lemma:hogwild_3} \\ 
%  [1ex] 
% \hline
%\end{tabular}
%\caption{Notation's Table}
%\label{table:1}
%
%%\todo[inline]{I would use $2 \kappa$ 
%%instead of $M$. $\kappa$ is a standard notation and we do not need to have $M$}
%\end{table}

\section{Review of Useful Theorems}\label{useful}

\begin{lem}[Generalization of the result in \cite{SVRG}]\label{lem_bound_diff_grad}
Let Assumptions \ref{ass_smooth} and \ref{ass_convex} hold. Then, $\forall w \in \mathbb{R}^d$, 
\begin{align}\label{eq:001}
\mathbb{E}[ \| \nabla f(w; \xi) - \nabla f(w_{*};\xi) \|^2 ] \leq 2L [ F(w) - F(w_{*})],   
\end{align}
where $\xi$ is a random variable, and $w_{*} = \arg \min_w F(w)$.  
\end{lem}

\begin{lem}[\cite{BertsekasSurvey}]\label{prop_supermartingale}
Let $Y_k$, $Z_k$, and $W_k$, $k = 0,1,\dots$, be three sequences of random variables and let $\{\mathcal{F}_k\}_{k\geq 0}$ be a filtration, that is, $\sigma$-algebras such that $\mathcal{F}_k \subset \mathcal{F}_{k+1}$ for all $k$. Suppose that: 
\begin{itemize}
\item The random variables $Y_k$, $Z_k$, and $W_k$ are nonnegative, and $\mathcal{F}_k$-measurable. 
\item For each $k$, we have $
\mathbb{E}[Y_{k+1} | \mathcal{F}_k] \leq Y_k - Z_k + W_k$. 
\item There holds, w.p.1, 
\begin{gather*}
\sum_{k=0}^{\infty} W_k < \infty. 
\end{gather*}
\end{itemize}
Then, we have, w.p.1,
\begin{gather*}
\sum_{k=0}^{\infty} Z_k < \infty \ \text{and} \ Y_k \to Y \geq 0. 
\end{gather*}
\end{lem}

\section{Proofs of Lemmas \ref{lem_bounded_secondmoment_04} and \ref{lem_bounded_secondmoment_04_new}}

\subsection{Proof of Lemma \ref{lem_bounded_secondmoment_04}}
\textbf{Lemma \ref{lem_bounded_secondmoment_04}}. \textit{Let Assumptions \ref{ass_smooth} and \ref{ass_convex} hold. Then, for $\forall w \in \mathbb{R}^d$,} 
\begin{gather*}
\mathbb{E}[\|\nabla f(w; \xi)\|^2] \leq  4 L [ F(w) - F(w_{*}) ] + N,
\end{gather*}
\textit{where $N = 2 \mathbb{E}[ \|\nabla f(w_{*}; \xi)\|^2 ]$; $\xi$ is a random variable, and $w_{*} = \arg \min_w F(w)$.}

\begin{proof}
Note that 
\begin{gather*}
\|a\|^2 = \|a - b + b\|^2 \leq 2\|a - b\|^2 + 2\|b\|^2, \tagthis{\label{eq_aaa001}} \\
\Rightarrow \frac{1}{2}\|a\|^2 - \|b\|^2 \leq \|a - b\|^2. \tagthis{\label{eq_aaa002}}
\end{gather*}

%From \eqref{eq:001}, we have
%\begin{align*}
%\mathbb{E}[\| \nabla f(w; \xi) - \nabla f(w_{*}; \xi) \|^2 ] \overset{\eqref{eq:001}}{\leq} 2 L [ F(w) - F(w_{*})]. \tagthis{\label{eq_aaa004}}
%\end{align*}

Hence, 
\begin{align*}
\frac{1}{2} \mathbb{E}[ \| \nabla f(w; \xi) \|^2 ] - \mathbb{E}[ \| \nabla f(w_{*}; \xi) \|^2 ] &= \mathbb{E} \left[ \frac{1}{2} \| \nabla f(w; \xi) \|^2 - \| \nabla f(w_{*}; \xi) \|^2 \right] \\ 
& \overset{\eqref{eq_aaa002}}{\leq} \mathbb{E}[\| \nabla f(w; \xi) - \nabla f(w_{*}; \xi) \|^2 ] \\
& \overset{\eqref{eq:001}}{\leq} 2 L [ F(w) - F(w_{*})] \tagthis{\label{eq_aaa005}}
\end{align*}

Therefore,
\begin{align*}
\mathbb{E}[\| \nabla f(w; \xi) \|^2] &\overset{\eqref{eq_aaa001}\eqref{eq_aaa005}}{\leq} 4 L [ F(w) - F(w_{*})] + 2 \mathbb{E}[ \| \nabla f(w_{*}; \xi) \|^2 ]. 
\qedhere 
\end{align*}
\end{proof}

\subsection{Proof of Lemma \ref{lem_bounded_secondmoment_04_new}}
\textbf{Lemma \ref{lem_bounded_secondmoment_04_new}}. \textit{Let Assumptions \ref{ass_stronglyconvex} and \ref{ass_smooth} hold. Then, for $\forall w \in \mathbb{R}^d$, 
\begin{gather*}
\mathbb{E}\|\nabla f(w; \xi)\|^2 \leq  4L \kappa [ F(w) - F(w_{*}) ] + N,
\end{gather*}
where $\kappa = \frac{L}{\mu}$ and $N = 2 \mathbb{E}[ \|\nabla f(w_{*}; \xi)\|^2 ]$; $\xi$ is a random variable, and $w_{*} = \arg \min_w F(w)$.}

\begin{proof}
Analogous to the proof of Lemma \ref{lem_bounded_secondmoment_04}, we have

Hence, 
\begin{align*}
\frac{1}{2} \mathbb{E}[ \| \nabla f(w; \xi) \|^2 ] - \mathbb{E}[ \| \nabla f(w_{*}; \xi) \|^2 ] &= \mathbb{E} \left[ \frac{1}{2} \| \nabla f(w; \xi) \|^2 - \| \nabla f(w_{*}; \xi) \|^2 \right] \\ 
& \overset{\eqref{eq_aaa002}}{\leq} \mathbb{E}[\| \nabla f(w; \xi) - \nabla f(w_{*}; \xi) \|^2 ] \\
& \overset{\eqref{eq:Lsmooth_basic}}{\leq} L^2 \| w - w_{*} \|^2 \\ 
& \overset{\eqref{eq:stronglyconvex_00}}{\leq} \frac{2L^2}{\mu}[F(w) - F(w_{*})] = 2 L \kappa [F(w) - F(w_{*})]. 
 \tagthis{\label{eq_aaa005_new}}
\end{align*}

Therefore,
\begin{align*}
\mathbb{E}[\| \nabla f(w; \xi) \|^2] \overset{\eqref{eq_aaa001}\eqref{eq_aaa005_new}}{\leq} 4 L \kappa [ F(w) - F(w_{*})] + 2 \mathbb{E}[ \| \nabla f(w_{*}; \xi) \|^2 ]. 
\end{align*}
\end{proof}

%\section{Proofs}\label{proof}

\section{Analysis for Algorithm \ref{sgd_algorithm}}\label{sec_analysis_sgd}

%\subsection{Proof of Lemma \ref{lem_bounded_secondmoment_04}}

In this Section, we provide the analysis of Algorithm \ref{sgd_algorithm} under Assumptions \ref{ass_stronglyconvex}, \ref{ass_smooth}, and \ref{ass_convex}.

We note that if $\{\xi_i\}_{i \geq 0}$ are i.i.d. random variables, then $\mathbb{E}[ \| \nabla f(w_{*}; \xi_0) \|^2 ] = \dots = \mathbb{E}[ \| \nabla f(w_{*}; \xi_t) \|^2 ]$. We have the following results for Algorithm \ref{sgd_algorithm}. 

%\subsection{Proof of Theorem \ref{thm_general_02_new_02}}

\textbf{Theorem \ref{thm_general_02_new_02}} (Sufficient condition for almost sure convergence). \textit{Let Assumptions \ref{ass_stronglyconvex},   \ref{ass_smooth} and \ref{ass_convex} hold. Consider Algorithm \ref{sgd_algorithm} with a stepsize sequence such that
\begin{align*}
0 < \eta_t \leq \frac{1}{2 L} \ , \ \sum_{t=0}^{\infty} \eta_t = \infty \ \text{and} \ \sum_{t=0}^{\infty} \eta_t^2 < \infty. 
\end{align*}
Then, the following holds w.p.1 (almost surely)
\begin{align*}
\| w_{t} - w_{*} \|^2 \to 0. 
\end{align*}}
\begin{proof}
Let $\mathcal{F}_{t} = \sigma(w_{0},\xi_{0},\dots,\xi_{t-1})$ be the $\sigma$-algebra generated by $w_{0},\xi_{0},\dots,\xi_{t-1}$, i.e., $\mathcal{F}_{t}$ contains all the information of $w_{0},\dots,w_{t}$. Note that $\mathbb{E}[\nabla f(w_{t}; \xi_t) | \mathcal{F}_{t}] = \nabla F(w_{t})$. By Lemma \ref{lem_bounded_secondmoment_04}, we have
\begin{gather*}
\mathbb{E}[\|\nabla f(w_{t}; \xi_t)\|^2 | \mathcal{F}_{t} ] \leq  4 L [ F(w_{t}) - F(w_{*}) ] + N, \tagthis \label{ineq:bounded_lemma3_new02}
\end{gather*}
where $N = 2 \mathbb{E}[ \| \nabla f(w_{*}; \xi_0) \|^2 ] = \dots = 2 \mathbb{E}[ \| \nabla f(w_{*}; \xi_t) \|^2 ]$ since $\{\xi_i\}_{i \geq 0}$ are i.i.d. random variables. Note that $w_{t+1} = w_{t} - \eta_t \nabla f(w_{t};\xi_t)$. Hence,
\begin{align*}
\mathbb{E}[\| w_{t+1} - w_{*} \|^2 | \mathcal{F}_{t}] & = \mathbb{E}[ \| w_{t} - \eta_t \nabla f(w_{t};\xi_t) - w_{*} \|^2 | \mathcal{F}_{t}] \\
& = \| w_{t} - w_{*} \|^2 - 2 \eta_t \langle \nabla F(w_{t})  , (w_{t} - w_{*}) \rangle + \eta_t^2 \mathbb{E}[\|\nabla f(w_{t}; \xi_t)\|^2 | \mathcal{F}_{t} ] \\
& \overset{\eqref{eq:stronglyconvex_00}\eqref{ineq:bounded_lemma3_new02}}{\leq} \| w_{t} - w_{*} \|^2 - \mu \eta_t \| w_{t} - w_{*} \|^2 - 2 \eta_t [ F(w_{t}) - F(w_{*}) ] + 4 L \eta_t^2 [ F(w_{t}) - F(w_{*}) ]  + \eta_t^2 N \\
& = \| w_{t} - w_{*} \|^2 - \mu \eta_t \| w_{t} - w_{*} \|^2 - 2 \eta_t ( 1 - 2 L \eta_t) [ F(w_{t}) - F(w_{*}) ]  + \eta_t^2 N \\
& \leq \| w_{t} - w_{*} \|^2 - \mu \eta_t \| w_{t} - w_{*} \|^2  + \eta_t^2 N.
\end{align*}
The last inequality follows since $0 < \eta_t \leq \frac{1}{2L}$. Therefore,
\begin{align*}
\mathbb{E}[\| w_{t+1} - w_{*} \|^2 | \mathcal{F}_{t}]  \leq \| w_{t} - w_{*} \|^2 - \mu \eta_t \| w_{t} - w_{*} \|^2  + \eta_t^2 N. \tagthis \label{main_ineq_sgd_new02}
\end{align*}
Since $\sum_{t=0}^{\infty} \eta_t^2 N < \infty$, 
we could apply Lemma \ref{prop_supermartingale}. Then, we have w.p.1,
\begin{gather*}
\| w_{t} - w_{*} \|^2 \to W \geq 0, \\ \text{and} \
\sum_{t=0}^{\infty} \mu \eta_t \| w_{t} - w_{*} \|^2 < \infty. 
\end{gather*}

We want to show that $\| w_{t} - w_{*} \|^2 \to 0$, w.p.1. Proving by contradiction, we assume that there exist $\epsilon > 0$ and $t_0$, s.t. $\| w_{t} - w_{*} \|^2 \geq \epsilon$ for $\forall t \geq t_0$. Hence, 
\begin{align*}
\sum_{t=0}^{\infty} \mu \eta_t \| w_{t} - w_{*} \|^2 \geq \mu \epsilon \sum_{t=0}^{\infty} \eta_t = \infty. 
\end{align*}
This is a contradiction. Therefore, $\|w_{t} - w_{*}\|^2 \to 0$ w.p.1. 
%Therefore, there exists a ball with center at $w_{*}$ and radius $$r = \max\{\|w_{0} - w_{*}\|,\dots,\|w_{t} - w_{*}\|\} < \infty$$ containing all iterations $w_{0},\dots,w_{t}$ w.p.1. 
\end{proof}

\textbf{Theorem \ref{thm_res_sublinear_new_02}}. \textit{Let Assumptions \ref{ass_stronglyconvex}, \ref{ass_smooth} and \ref{ass_convex} hold. Let $E = \frac{2\alpha L}{\mu}$ with $\alpha=2$. Consider Algorithm \ref{sgd_algorithm} with a stepsize sequence such that $\eta_t = \frac{\alpha}{\mu(t+E)} \leq \eta_0=\frac{1}{2L}$. The expectation} $\mathbb{E}[\|w_{t} - w_{*}\|^2]$ \textit{is at most}
$$ \frac{4\alpha^2 N}{\mu^2} 
\frac{1 }{(t-T+E)} $$
\textit{for} $t\geq T =\frac{4L}{\mu}\max \{ \frac{L\mu}{N} \|w_{0} - w_{*}\|^2, 1\} - \frac{4L}{\mu}$.

\begin{proof}
Using the beginning of the proof of Theorem \ref{thm_general_02_new_02}, taking the expectation to \eqref{main_ineq_sgd_new02}, with $0 < \eta_t \leq \frac{1}{2L}$, we have
\begin{align*}
\mathbb{E}[\| w_{t+1} - w_{*} \|^2 ]  \leq (1 - \mu \eta_t) \mathbb{E}[ \| w_{t} - w_{*} \|^2]  + \eta_t^2 N.   
\end{align*} 

We first show that
\begin{equation} \mathbb{E}[\|w_{t} - w_{*}\|^2]\leq \frac{N}{\mu^2} G
\frac{1 }{(t+E)}, \label{eqGGG} \end{equation}
where 
$G=\max\{I,J\}$, and 
\begin{align*}
I &= \frac{E \mu^2}{N} \mathbb{E}[\|w_{0} - w_{*}\|^2]>0, \\
J &= \frac{\alpha^2}{\alpha - 1}>0.
\end{align*}

We use mathematical induction to prove (\ref{eqGGG}) (this trick is based on the idea from \cite{bottou2016optimization}). Let $t = 0$, we have 
\begin{align*}
\mathbb{E}[\|w_{0} - w_{*}\|^2] %= \|w_{0} - w_{*}\|^2 
\leq \frac{NG}{\mu^2 E},  
\end{align*}
%we have $\eta_0 = \frac{\alpha}{\mu E}$, so
%\begin{align*}
%\mathbb{E}[\|w_{1} - w_{*}\|^2]  \leq \left(1 - \frac{\alpha}{E} \right)\|w_{0} - w_{*}\|^2 + \frac{\alpha^2 N}{\mu^2 E^2}  \leq \frac{NG}{\mu^2 (E+1)},  
%\end{align*}
which is obviously true since 
$
G \geq \frac{E\mu^2}{N} \|w_{0} - w_{*}\|^2.
$

Suppose it is true for $t$, we need to show that it is also true for $t+1$. We have
\begin{align*}
\mathbb{E}[\|w_{t+1} - w_{*}\|^2]  & \leq \left(1 - \frac{\alpha}{t+E} \right) \frac{NG}{\mu^2(t+E)} + \frac{\alpha^2 N}{\mu^2(t+E)^2} \\
& = \left(\frac{t +E - \alpha}{\mu^2(t+E)^2} \right) NG + \frac{\alpha^2N}{\mu^2(t+E)^2} \\
& = \left(\frac{t +E - 1}{\mu^2(t+E)^2} \right) NG  - \left(\frac{\alpha - 1}{\mu^2(t+E)^2} \right) NG + \frac{\alpha^2 N}{\mu^2(t+E)^2}.
\end{align*}
Since
$
G \geq \frac{\alpha^2}{\alpha - 1},
$
$$ - \left(\frac{\alpha - 1}{\mu^2(t+E)^2} \right) NG + \frac{\alpha^2 N}{\mu^2(t+E)^2}\leq 0.$$
This implies
\begin{align*}
\mathbb{E}[\|w_{t+1} - w_{*}\|^2]  & \leq \left(\frac{t+E - 1}{\mu^2(t+E)^2} \right) NG \\
& = \left(\frac{(t+E)^2 - 1}{(t+E)^2} \right) \frac{NG}{\mu^2(t +E+ 1)} \\
& \leq\frac{NG}{\mu^2(t +E+ 1)}. 
\end{align*}
This proves (\ref{eqGGG}) by induction in $t$.

Notice that the induction proof of (\ref{eqGGG}) holds more generally for $E\geq \frac{2\alpha L}{\mu}$ with $\alpha>1$ (this is sufficient for showing $\eta_t \leq \frac{1}{2L}$. In this more general interpretation we can see that the convergence rate is minimized for $I$ minimal, i.e., $E=\frac{2\alpha L}{\mu}$ and for this reason we have fixed $E$ as such in the theorem statement.

Notice that 
$$G=\max\{I,J\} = \max \{ \frac{2 \alpha L \mu}{N} \mathbb{E}[\|w_{0} - w_{*}\|^2], \frac{\alpha^2}{\alpha - 1} \}.$$
We choose $\alpha=2$ such that $\eta_t$ only depends on known parameters $\mu$ and $L$. For this $\alpha$ we obtain
$$G = 4 \max \{\frac{L \mu}{N} \mathbb{E}[\|w_{0} - w_{*}\|^2], 1\}.$$

For
$T =\frac{4L}{\mu}\max \{ \frac{L\mu}{N} \mathbb{E}[\|w_{0} - w_{*}\|^2], 1\} - \frac{4L}{\mu}$, 
 we have that
according to (\ref{eqGGG})
\begin{eqnarray}
 \frac{L \mu}{N} \mathbb{E}[\|w_{T} - w_{*}\|^2]&\leq& \frac{L \mu}{N}\frac{N}{\mu^2} 
\frac{G }{(T+E)} \nonumber \\
&=& \frac{L}{\mu} \frac{ 4 \max \{\frac{L \mu}{N} \mathbb{E}[\|w_{0} - w_{*}\|^2], 1\}}
{\frac{4L}{\mu}\max \{\frac{L\mu}{N} \mathbb{E}[\|w_{0} - w_{*}\|^2], 1\} } =
%\frac{  \max \{\frac{L \mu}{N} \mathbb{E}[\|w_{0} - w_{*}\|^2], 1\}}
%{\frac{ L\mu}{N} \mathbb{E}[\|w_{0} - w_{*}\|^2]}\leq 
1. \label{eqTT} \end{eqnarray}
Applying (\ref{eqGGG})with $w_T$ as starting point rather than $w_0$ gives, for $t\geq \max\{T,0\}$,
$$\mathbb{E}[\|w_{t} - w_{*}\|^2]\leq \frac{N}{\mu^2} G
\frac{1 }{(t-T+E)},$$
where $G$ is now equal to
$$ 4 \max \{ \frac{L \mu}{N} \mathbb{E}[\|w_{T} - w_{*}\|^2], 1\},$$
which equals $4$, see (\ref{eqTT}). For any given $w_0$, we prove the theorem. 
\end{proof}

%\input{Appendix_hogwild_2.tex}

%\input{hogwild.tex}
 
%\input{Appendix_hogwild_3.tex}  

%\input{sparseHogWild.tex}

%\input{hogwild_UCONN_New.tex} 
%\twocolumn

%\clearpage
%\onecolumn

\section{Analysis for Algorithm~\ref{HogWildAlgorithm}}
\label{sec:Hogwild_insconsistent_read_write}

\subsection{Recurrence and Notation}

We introduce the following notation: For each $\xi$, we define $D_\xi \subseteq \{1,\ldots, d\}$ as the set of possible non-zero positions in a vector of the form $\nabla f(w;\xi)$ for some $w$. We consider a fixed mapping from $u\in U$ to subsets $S^{\xi}_u\subseteq D_\xi$ for each possible $\xi$. In our notation we also let $D_\xi$ represent the diagonal $d\times d$ matrix with ones exactly at the positions corresponding to $D_\xi$ and with zeroes elsewhere. Similarly, $S^{\xi}_u$ also denotes a diagonal matrix with ones at the positions corresponding to $D_\xi$.

We will use a probability distribution $p_\xi(u)$ to indicate how to randomly select a matrix $S^\xi_u$. We choose the matrices $S^\xi_u$ and distribution $p_\xi(u)$ so that there exist $d_\xi$ such that
\begin{equation} d_\xi \mathbb{E}[S^\xi_u | \xi] = D_\xi, \label{eq:Sexp} \end{equation}
where the expectation is over $p_\xi(u)$.

We will restrict ourselves to choosing {\em non-empty} sets  $S^\xi_u$ that partition $D_\xi$ in $D$ approximately equally sized sets together with uniform distributions $p_\xi(u)$ for some fixed $D$. So, if $D\leq |D_\xi|$, then sets have sizes $\lfloor |D_\xi|/D \rfloor$ and $\lceil |D_\xi|/D \rceil$. For the special case $D>|D_\xi|$ we have exactly $|D_\xi|$ singleton sets of size $1$ (in our definition we only use non-empty sets).

For example, for $D=\bar{\Delta}$, where 
$$ \bar{\Delta} = \max_\xi \{ |D_\xi|\}$$
represents the maximum number of non-zero positions in any gradient computation $f(w;\xi)$, we have that for all $\xi$, there are exactly $|D_\xi|$ singleton sets $S^\xi_u$ representing each of the elements in $D_\xi$. Since  $p_\xi(u)= 1/|D_\xi|$ is the uniform distribution, we have $\mathbb{E}[S^\xi_u | \xi] = D_\xi / |D_\xi|$, hence, $d_\xi = |D_\xi|$.
As another example at the other extreme, for $D=1$, we have exactly one set $S^\xi_1=D_\xi$ for each $\xi$. Now $p_\xi(1)=1$ and we have $d_\xi=1$.

We define the parameter
$$ \bar{\Delta}_D \eqdef D \cdot \mathbb{E}[\lceil |D_\xi|/D \rceil],$$
where the expectation is over $\xi$.
We use $\bar{\Delta}_D$ in the leading asymptotic term for the convergence rate in our main theorem. We observe that
$$ \bar{\Delta}_D \leq \mathbb{E}[|D_\xi|] + D-1$$
and
$\bar{\Delta}_D \leq \bar{\Delta}$ with equality for $D=\bar{\Delta}$.

For completeness we define
\begin{equation}
\Delta \eqdef 
\max_{i}
\Prob \left(   i \in  D_\xi  \right). \nonumber
\end{equation}
Let us remark, that $\Delta \in (0,1]$ measures the probability of collision.
Small $\Delta$ means that 
there is a small chance that the support of two random realizations of $\nabla f(w;\xi)$
will have an intersection.
On the other hand, $\Delta = 1$ means that almost surely, the support of two stochastic gradients will have non-empty intersection.

With this definition of $\Delta$ it is an easy exercise to show that 
for iid $\xi_1$ and $\xi_2$ in a finite-sum setting (i.e., $\xi_i$ and $\xi_2$ can only take on a finite set of possible values) we have
\begin{align*}
&\Exp[|\ve{ \nabla f(w_1; \xi_1) }{\nabla f(w_2; \xi_2)}|] 
\\& \leq 
\frac{\sqrt \Delta}2
\left(
\Exp[\|  \nabla f(w_1; \xi_1) \|^2] + \Exp[\|\nabla f(w_2; \xi_2)\|^2]
\right)
\tagthis \label{Masdfasdfasfa}
\end{align*}
(see Proposition 10 in \cite{Leblond2018}).
We notice that in the non-finite sum setting we can use the property that for any two vectors $a$ and $b$, $\langle a, b\rangle \leq (\|a\|^2 + \|b\|^2)/2$ and this proves (\ref{Masdfasdfasfa}) with $\Delta$ set to $\Delta=1$. In our asymptotic analysis of the convergence rate, we will show how $\Delta$ plays a role in non-leading terms -- this, with respect to the leading term, it will not matter whether we use $\Delta=1$ or $\Delta$ equal the probability of collision (in the finite sum case).

%\todo[inline]{We can explain what this $\Delta$ is in a finite-sum setting, and maybe linking to \cite{Leblond2018}?????}

%We will assume two constants $D$ and $K$ with 
%$$ d_\xi \leq D \mbox{ and } |S^\xi_u|\leq K$$
%for all $\xi$ and $u$.

%For example, we can chose each set $S^\xi_u$ to be a singleton set with $D_\xi = \cup_u S^\xi_u$ and let $p_\xi(u)$ be the uniform distribution. This gives $K=1$ and $d_\xi = |D_\xi|$ implying $D$ is the maximum number $\Delta$ of non-zero positions in a gradient computation $f(w;\xi)$. At the other extreme we have $S^\xi_1=D_\xi$ with $p_\xi(1)=1$. In this case $|S^\xi_1|=|D_\xi|$ implying that $K$ is the maximum number $\Delta$ of non-zero positions in a gradient computation $f(w;\xi)$. Now $d_\xi=1$, hence $D=1$.

%If we partition $D_\xi$ into $\lceil |D_\xi|/K\rceil$ subsets $S^\xi_u$ of at most size $K$ and again let $p_\xi(u)$ be the uniform distribution, then $d_\xi= \lceil |D_\xi|/K\rceil$ and $D$ is equal to the number $\Delta$ of non-zero positions in a gradient computation $f(w;\xi)$ divided by $K$. This shows that we can realize
%$$ D =  \lceil \Delta/K\rceil $$
%for any $K$.

%\subsection{Recursion}

We have
\begin{equation}
 w_{t+1} = w_t - \eta_t d_{\xi_t}  S^{\xi_t}_{u_t} \nabla f(\hat{w}_t;\xi_t),\label{eqw}
 \end{equation}
where  $\hat{w}_t$ represents the vector used in computing the gradient $\nabla f(\hat{w}_t;\xi_t)$ and whose entries have been read (one by one)  from  an aggregate of a mix of  previous updates that led to $w_{j}$, $j\leq t$.
% where $\hat{w}_t$ represents the vector used in computing the gradient and whose positions have been read from a mix of  previous updates that led to $w_{j}$, $j\leq t$. 
Here, we assume that
\begin{itemize}
\item updating/writing to vector positions is atomic, reading vector positions is atomic, and
\item there exists a ``delay'' $\tau$ such that, for all $t$, vector $\hat{w}_t$ includes all the updates up to and including those made during the $(t-\tau)$-th iteration (where (\ref{eqw}) defines the $(t+1)$-st iteration).
\end{itemize}
Notice that we do {\bf not assume consistent reads and writes of vector positions}. We only assume that up to a ``delay'' $\tau$ all writes/updates are included in the values of positions that are being read.

According to our definition of $\tau$, in (\ref{eqw}) vector $\hat{w}_t$ represents an inconsistent read with entries that contain all of the updates made during the $1$st to $(t-\tau)$-th iteration. Furthermore each entry in $\hat{w}_t$ includes some of the updates made during the $(t-\tau+1)$-th iteration up to $t$-th iteration. Each entry includes its own subset of updates because writes are inconsistent.  We model this by ``masks'' $\Sigma_{t,j}$ for $t-\tau\leq j\leq t-1$. A mask $\Sigma_{t,j}$ is a diagonal 0/1-matrix with the 1s expressing which of the entry updates made in the $(j+1)$-th iteration are included in $\hat{w}_t$.
That is,
 \begin{equation}
\label{eq:what_wrhot_1}
\hat{w}_t = w_{t-\tau} - \sum_{j=t-\tau}^{t-1} \eta_j d_{\xi_j} \Sigma_{t,j}  S^{\xi_j}_{u_j} \nabla  f(\hat{w}_j;\xi_j)  .
\end{equation}

%, i.e., the last read positions stem from some $w_i$ with $j\leq t-1$). Let $G_{t,i}$ represent the diagonal matrix with ones corresponding to positions $p$ in $\hat{w}_t$ whose values correspond to $w_{i}$ but not to $w_{i-1}$. In other words the value of $\hat{w}_t$ for such a  position $p$ must correspond to the value of $w_{i}$ at position $p$ but not the value of $w_{i-1}$. This means that, for $t-\rho_t+1 \leq i\leq t$,
%$$ G_{t,i} \hat{w}_t = G_{t,j} [w_{t-\rho_t} - \sum_{j=t-\rho_t}^{i-1} \eta_j d_{\xi_j}  S^{\xi_j}_{u_j} \nabla f(\hat{w}_j;\xi_j) ]$$
%and since the reading started at logical time $t-\rho_t$
%$$ (\sum_{i=1}^{t-\rho_t} G_{t,i} ) \hat{w}_t = (\sum_{i=1}^{t-\rho_t} G_{t,i} ) w_{t-\rho_t} .$$
%Also, notice that the sum of all matrices $G_{t,i}$ is equal to the identity matrix (the matrices correspond to a partition of the set $\{1,\ldots, d\}$).
%Therefore, adding all equations yields
%\begin{eqnarray*}
% \hat{w}_t &=& w_{t-\rho_t} - \sum_{i=t-\rho_t+1}^t  G_{t,j} \sum_{j=t-\rho_t}^{i-1} \eta_j d_{\xi_j}  S^{\xi_j}_{u_j} \nabla  f(\hat{w}_j;\xi_j)  \\
% &=& w_{t-\rho_t} - \sum_{j=t-\rho_t}^{t-1}(\sum_{i=j+1}^t G_{t,j}) \eta_j d_{\xi_j}  S^{\xi_j}_{u_j}  \nabla f(\hat{w}_j;\xi_j)  .
% \end{eqnarray*}
% We define
% $$ \Sigma_{t,j} = \sum_{i=j+1}^t G_{t,j}$$
% and conclude
% \begin{equation}
%\label{eq:what_wrhot_1}
%\hat{w}_t = w_{t-\rho_t} - \sum_{j=t-\rho_t}^{t-1} \eta_j d_{\xi_j} \Sigma_{t,j}  S^{\xi_j}_{u_j} \nabla  f(\hat{w}_j;\xi_j)  .
%\end{equation}

Notice that the recursion (\ref{eqw}) implies
\begin{equation}
\label{eq:wt_wrhot_1}
w_t = w_{t-\tau} - \sum_{j=t-\tau}^{t-1} \eta_j d_{\xi_j}  S^{\xi_j}_{u_j} \nabla  f(\hat{w}_j;\xi_j) .
\end{equation}
By combining~\eqref{eq:wt_wrhot_1} and~\eqref{eq:what_wrhot_1} we obtain
\begin{equation}
\label{eq:wt_what_2}
w_t-\hat{w}_t = - \sum_{j=t-\tau}^{t-1}  \eta_j d_{\xi_j} (I-\Sigma_{t,j})  S^{\xi_j}_{u_j}  \nabla f(\hat{w}_j;\xi_j) ,
\end{equation}
where $I$ represents the identity matrix.

%See~\eqref{eq:hogwild_update_inconsistent_rw}, vector $w_t$ satisfies
%\begin{equation}
%\label{eq:wt_wrhot_1}
%w_t = w_{t-\rho_t} - \sum_{j=t-\rho_t}^{t-1} \eta_j d_j \langle \nabla f(\hat{w}_j;\xi_j),e[s_j] \rangle e[s_j].
%\end{equation}
%Here, vector $\hat{w}_t$ represents an inconsistent read starting from  moment $t-\rho_t$. This means that we can represent vector $\hat{w}_t$ by
%\begin{equation}
%\label{eq:what_wrhot_1}
%\hat{w}_t = w_{t-\rho_t} - \sum_{j=t-\rho_t}^{t-1}  \eta_j d_j \langle \nabla f(\hat{w}_j;\xi_j),e[s_j] \rangle e[s_j] \sigma_{t,j},
%\end{equation}
%for some $\sigma_{t,j} \in \{0,1\}$. By combining~\eqref{eq:wt_wrhot_1} and~\eqref{eq:what_wrhot_1} we obtain
% \begin{equation}
%\label{eq:wt_what_2}
%w_t-\hat{w}_t = - \sum_{j=t-\rho_t}^{t-1}  \eta_j d_j \langle \nabla f(\hat{w}_j;\xi_j),e[s_j] \rangle e[s_j] (1-\sigma_{t,j}). 
%\end{equation}

\subsection{Main Analysis}

We first derive a couple lemmas which will help us deriving our main bounds. In what follows let Assumptions \ref{ass_stronglyconvex}, \ref{ass_smooth}, \ref{ass_convex} and  \ref{ass_tau} hold for all lemmas. We define 
$$\mathcal{F}_t=\sigma(w_0,\xi_1,u_1,\sigma_1,\dotsc, \xi_{t-1},u_{t-1},\sigma_{t-1}), $$
where
$$\sigma_{t-1} = (\Sigma_{t,t-\tau}, \ldots, \Sigma_{t,t-1}).$$
When we subtract $\tau$ from, for example, $t$ and write $t-\tau$, we will actually mean $\max\{t-\tau, 0\}$.

 \begin{lem} \label{lem:expect}  %Let $d_t$ denote the cardinality of $\mathcal{D}_t$ and assume $d_t\leq D$. 
 We have
 $$\mathbb{E}[\| d_{\xi_t}S^{\xi_t}_{u_t} \nabla  f(\hat{w}_t;\xi_t)  \|^2| \mathcal{F}_t, \xi_t]\leq  D \|\nabla f(\hat{w}_t;\xi_t) \|^2$$
% $$\mathbb{E}[\|\eta_t d_t \langle \nabla f(\hat{w}_t;\xi_t)-\nabla f(w_*;\xi_t) ,e[s_t] \rangle e[s_t] \|^2| \mathcal{F}_t, \xi_t, \sigma_t]\leq \eta^2_t D \|\nabla f(\hat{w}_t;\xi_t)-\nabla f(w_*;\xi_t) \|^2, $$
 and
 $$\mathbb{E}[d_{\xi_t}S^{\xi_t}_{u_t} \nabla  f(\hat{w}_t;\xi_t) | \mathcal{F}_t ] 
= \nabla F(\hat{w}_t).$$
 \end{lem}
 
 \begin{proof}
 For the first bound, if we take the expectation of $\| d_{\xi_t}S^{\xi_t}_{u_t} \nabla  f(\hat{w}_t;\xi_t)  \|^2$ with respect to $u_t$, then we have (for vectors $x$ we denote the value if its $i$-th position by $[x]_i$)
 \begin{align*}
& \mathbb{E}[\|d_{\xi_t}S^{\xi_t}_{u_t} \nabla  f(\hat{w}_t;\xi_t)  \|^2| \mathcal{F}_t, \xi_t] = 
 d_{\xi_t}^2 \sum_u p_{\xi_t}(u)  \|S^{\xi_t}_{u} \nabla  f(\hat{w}_t;\xi_t)  \|^2 = 
 d_{\xi_t}^2 \sum_u p_{\xi_t}(u)  \sum_{i\in S^{\xi_t}_{u}} [\nabla f(\hat{w}_t;\xi_t)]_i^2 \\
&= 
 d_{\xi_t} \sum_{i\in D_{\xi_t}} [\nabla f(\hat{w}_t;\xi_t)]_i^2 =
 d_{\xi_t} \| f(\hat{w}_t;\xi_t) \|^2 \leq  D \|\nabla f(\hat{w}_t;\xi_t) \|^2,
\end{align*}
where the transition to the second line follows from (\ref{eq:Sexp}).
%Since $\nabla f(.;\xi_j)$ has support $D_j$, we can use the same argument to derive the second bound.

For the second bound, if we take the expectation of $ d_{\xi_t}S^{\xi_t}_{u_t} \nabla  f(\hat{w}_t;\xi_t)$ wrt $u_t$, then we have:
\begin{align*}
\mathbb{E}[ d_{\xi_t}S^{\xi_t}_{u_t} \nabla  f(\hat{w}_t;\xi_t) | \mathcal{F}_t, \xi_t] &=
 d_{\xi_t} \sum_u p_{\xi_t}(u) S^{\xi_t}_{u}  \nabla f(\hat{w}_t;\xi_t) = D_{\xi_t} \nabla  f(\hat{w}_t;\xi_t)= 
    \nabla f(\hat{w}_t;\xi_t),
\end{align*}
and this can be used to derive
\begin{align*}
\mathbb{E}[ d_{\xi_t}S^{\xi_t}_{u_t}  f(\hat{w}_t;\xi_t) | \mathcal{F}_t]
=
\mathbb{E}[\mathbb{E}[d_{\xi_t}S^{\xi_t}_{u_t}  f(\hat{w}_t;\xi_t) | \mathcal{F}_t, \xi_t] | \mathcal{F}_t] &=  \nabla F(\hat{w}_t).  %\tagthis \label{eq:Ewhat1}
\end{align*}
\end{proof}

As a consequence of this lemma we derive a bound on the expectation of  $\|w_t- \hat{w}_t \|^2$.

%%%%%%%%%%%% LEMMA below does not use Grad[ f(w_*;xi_j)] has support D_j -- this means that we can use this in another analysis where we define support to mean the entries are ``small enough'' which could allow diminishing Delta and speeding up convergence. I expect this to have a minimal effect. 
%%%%%%%%%%%%%%%%%%
\begin{lem}
\label{lemma:hogwild_21} The expectation of $\|w_t- \hat{w}_t \|^2$ is at most
%Assume  $\rho_t\leq R$ for all $t$. Then,
$$\mathbb{E}[\|w_t- \hat{w}_t \|^2 ] \leq  (1+\sqrt{\Delta}\tau) D \sum_{j=t-\tau}^{t-1} \eta_j^2 (2L^2 \mathbb{E}[\| \hat{w}_j -w_{*} \|^2] +N). $$
\end{lem}

\begin{proof}
As shown in~\eqref{eq:wt_what_2}, 
$$
w_t-\hat{w}_t = - \sum_{j=t-\tau}^{t-1}  \eta_j d_{\xi_j} (I-\Sigma_{t,j})  S^{\xi_j}_{u_j} \nabla  f(\hat{w}_j;\xi_j) .
%- \sum_{j=t-\rho_t}^{t-1} \eta_j d_j \langle \nabla f(\hat{w}_j;\xi_j),e[s_j] \rangle e[s_j] (1-\sigma_j).
$$
This can be used to derive an expression for the square of its norm:
\begin{eqnarray*}
\|w_t- \hat{w}_t \|^2 
&=& \|  \sum_{j=t-\tau}^{t-1}  \eta_j d_{\xi_j} (I-\Sigma_{t,j})  S^{\xi_j}_{u_j}  \nabla  f(\hat{w}_j;\xi_j) \|^2 \\
%&\leq &  \|  \sum_{j=t-\rho_t}^{t-1}  \eta_j d_{\xi_j}    (I-\Sigma_{t,j}) S^{\xi_j}_{u_j}  f(\hat{w}_j;\xi_j) \|^2 \\
&=& \sum_{j=t-\tau}^{t-1} \| \eta_j d_{\xi_j}    (I-\Sigma_{t,j}) S^{\xi_j}_{u_j} \nabla   f(\hat{w}_j;\xi_j) \|^2   \\
&& + \sum_{i\neq j \in \{t-\tau, \ldots, t-1\}}
\ve{ \eta_j d_{\xi_j}   (I-\Sigma_{t,j}) S^{\xi_j}_{u_j} \nabla   f(\hat{w}_j;\xi_j)}{ \eta_i d_{\xi_i}   (I-\Sigma_{t,j}) S^{\xi_i}_{u_i}  \nabla  f(\hat{w}_i;\xi_i) }.
%   (\eta_j d_j \langle \nabla f(\hat{w}_j;\xi_j),e[s_j] \rangle) \cdot (\eta_i d_i \langle \nabla f(\hat{w}_i;\xi_i),e[s_i] \rangle) \cdot D_{s_i=s_j}, 
\end{eqnarray*}
%where $D_{s_i=s_j}$ is equal to $1$ if $s_i=s_j$ and equal to $0$ otherwise. Now notice that for any real numbers $a$ and $b$, $ab\leq (a^2+b^2)/2$. 
%With this definition of $D$ it is an easy exercise to show that for iid $\xi_1$ and $\xi_2$ we have

%Now notice that, for any two vectors $a$ and $b$, $\langle a, b\rangle \leq (\|a\|^2 + \|b\|^2)/2$.
%real numbers $a$ and $b$, $ab\leq (a^2+b^2)/2$. Therefore, $\sum_{i=1}^K a_ib_i\leq \sum_{i=1}^K (a_i^2+b_i^2)/2$. 
Applying (\ref{Masdfasdfasfa}) to the inner products implies
%
%Since $(\sum_{i=1}^K a_ib_i)^2 = \sum_{i=1}^K (a_i b_i)^2 + \sum_{
%
%Since vectors $\eta_j d_{\xi_j}   S^{\xi_j}_{u_j}  f(\hat{w}_j;\xi_j)$ and $\eta_i d_{\xi_i}   S^{\xi_i}_{u_i}  f(\hat{w}_i;\xi_i)$ are identically distributed with  non-zero positions in  $S^{\x_j)_{u_j}$ and $S^{x_i}_{u_i}$, i.e., at most $K$ non-zero positions by (\ref{DK}), we have
%
%\mathbb{E}[\langle \eta_j d_{\xi_j}   S^{\xi_j}_{u_j}  f(\hat{w}_j;\xi_j), \eta_i d_{\xi_i}   S^{\xi_i}_{u_i}  f(\hat{w}_i;\xi_i) \rangle]
%\leq \frac{\sqrt{K}}{2} \left(\mathbb{E}[\| \eta_j d_{\xi_j}   S^{\xi_j}_{u_j}  f(\hat{w}_j;\xi_j) \|^2]
%+ \mathbb{E}[\| \eta_i d_{\xi_i}   S^{\xi_i}_{u_i}  f(\hat{w}_i;\xi_i) \|^2]
%
%\begin{align*}
%&\Exp[|\ve{ g(w_1, \xi_1) }{g(w_2, \xi_2)}|] 
%\\& \leq 
%\frac{\sqrt D}2
%\left(
%\Exp[\|g(w_1, \xi_1)\|^2] + \Exp[\|g(w_2, \xi_2)\|^2]
%\right)
%\tagthis \label{asdfasdfasfa}
%\end{align*}
%(see Proposition 10 in \cite{Leblond2018}).
%Using this implies
\begin{eqnarray}
\|w_t- \hat{w}_t \|^2 
&\leq&  \sum_{j=t-\tau}^{t-1} \| \eta_j d_{\xi_j}   (I-\Sigma_{t,j}) S^{\xi_j}_{u_j} \nabla   f(\hat{w}_j;\xi_j) \|^2  \nonumber  \\
&& + \sum_{i\neq j \in \{t-\tau, \ldots, t-1\}}
[\| \eta_j d_{\xi_j}  (I-\Sigma_{t,j})  S^{\xi_j}_{u_j}  \nabla  f(\hat{w}_j;\xi_j)\|^2 + \| \eta_i d_{\xi_i}   (I-\Sigma_{t,j})  S^{\xi_i}_{u_i}  \nabla  f(\hat{w}_i;\xi_i) \|^2]\sqrt{\Delta}/2 \nonumber  \\
%&& + \sum_{i\neq j \in \{t-\rho_t, \ldots, t-1\}} [ (\eta_j d_j \langle \nabla f(\hat{w}_j;\xi_j),e[s_j] \rangle)^2 + (\eta_i d_i \langle \nabla f(\hat{w}_i;\xi_i),e[s_i] \rangle)^2]/2 \\
%&\leq &
%\sum_{j=t-R}^{t-1} \| \eta_j d_{\xi_j}   (I-\Sigma_{t,j}) S^{\xi_j}_{u_j}  \nabla  f(\hat{w}_j;\xi_j) \|^2  \nonumber   \\
%&& + \sum_{i\neq j \in \{t-R, \ldots, t-1\}}[\| \eta_j d_{\xi_j}  (I-\Sigma_{t,j})  S^{\xi_j}_{u_j}  \nabla  f(\hat{w}_j;\xi_j)\|^2 + \| \eta_i d_{\xi_i}    (I-\Sigma_{t,j}) S^{\xi_i}_{u_i} \nabla   f(\hat{w}_i;\xi_i) \|^2]\sqrt{\Delta}/2 \nonumber  \\
&=& (1+\sqrt{\Delta}\tau) \sum_{j=t-\tau}^{t-1} \| \eta_j d_{\xi_j}    (I-\Sigma_{t,j}) S^{\xi_j}_{u_j}  \nabla  f(\hat{w}_j;\xi_j) \|^2  \nonumber  \\
&\leq&   (1+\sqrt{\Delta}\tau) \sum_{j=t-\tau}^{t-1} \eta_j^2  \|  d_{\xi_j}    S^{\xi_j}_{u_j}  \nabla  f(\hat{w}_j;\xi_j) \|^2 . \nonumber %\label{eq:reason}
\end{eqnarray}
Taking expectations shows
$$\mathbb{E}[\|w_t- \hat{w}_t \|^2 ] \leq   (1+\sqrt{\Delta}\tau) \sum_{j=t-\tau}^{t-1} \eta_j^2  \mathbb{E}[ \| d_{\xi_j}   S^{\xi_j}_{u_j}  \nabla  f(\hat{w}_j;\xi_j) \|^2].$$
%  d_j \langle \nabla f(\hat{w}_j;\xi_j),e[s_j] \rangle)^2].$$
Now, we can apply Lemma \ref{lem:expect}: We first take the expectation over $u_j$ and this shows
$$\mathbb{E}[\|w_t- \hat{w}_t \|^2 ] \leq  (1+\sqrt{\Delta}\tau) \sum_{j=t-\tau}^{t-1} \eta_j^2 D \mathbb{E}[ \|\nabla f(\hat{w}_j;\xi_j) \|^2 ].$$
From Lemma \ref{lem_bounded_secondmoment_04} we infer
\begin{equation} \mathbb{E}[ \|\nabla f(\hat{w}_j;\xi_j) \|^2 ] \leq 4L\mathbb{E}[F(\hat{w}_j)-F(w_{*})] +N \label{LN}
\end{equation}
%The L^2 term multiplied by kappa
%{ineq:bounded_lemma3_hogwild} the L term has a kappa, Since $\eta_t \leq \frac{1}{4LD}$, $-2\eta_t(1-4L\eta_tD) [F(w_{t}) - F(w_{*})]\leq 0$ Lemma 7 with <= ..
% (31) multiplied kappa in L^2 term and in Lemma 9 wher also E>= 4L\tau ..  and in (40); theorem does not change except for E>=
% For the non-convex case $L$ in (\ref{LN}) must be replaced by $L\kappa$ and as a result $L^2$ in Lemma \ref{lemma:hogwild_21} must be replaced by $L^2\kappa$.  Also $L$ in (\ref{ineq:bounded_lemma3_hogwild}) must be replaced by $L\kappa$. We now require that $\eta_t \leq \frac{1}{4L\kappa D}$ so that $-2\eta_t(1-4L\kappa \eta_tD) [F(w_{t}) - F(w_{*})]\leq 0$. This leads to Lemma \ref{lemma:hogwild_4_new1} where no changes are needed except requiring  $\eta_t\leq \frac{1}{4L\kappa D}$. The changes in Lemmas \ref{lemma:hogwild_21} and \ref{lemma:hogwild_4_new1}  lead to a Lemma \ref{lemma:hogwild_4_new1} where we require $E\geq \frac{4L\kappa \alpha D}{\mu}$ and where in the bound of the expectation $L^2$ must be replaced by $L^2\kappa$. This perculates through to inequality (40)  with a similar change finally leading to Theorem 3 where we only need to strengthen the condition on $E$ to $E\geq \frac{4L\kappa \alpha D}{\mu}$.
and by $L$-smoothness, see Equation \ref{eq:Lsmooth} with $\nabla F(w_*) =0$,
$$F(\hat{w}_j)-F(w_*) \leq \frac{L}{2} \| \hat{w}_j -w_{*} \|^2.$$
Combining the above inequalities proves the lemma.
\end{proof}

 Together with the next lemma we will be able to start deriving a recursive inequality from which we will be able to derive a bound on the convergence rate.

\begin{lem}\label{lemma:hogwild_1_new1}
Let $0 < \eta_t \leq \frac{1}{4LD}$  for all $t \geq 0$. Then,
\begin{align*}
 \mathbb{E}[\| w_{t+1} - w_{*} \|^2 |\mathcal{F}_t] \leq \left(1 - \frac{\mu \eta_t}{2} \right) \| w_{t} - w_{*} \|^2 + [(L+\mu)\eta_t + 2L^2 \eta_t^2D] \| \hat{w}_t - w_{t}  \|^2 + 2 \eta^2_t D N. 
\end{align*}
%where 
%\begin{align*}
%N = 2 \mathbb{E} [ \|\nabla f(w_{*}; \xi_1)\|^2 ]. 
%\end{align*}
\end{lem}

\begin{proof}
Since $w_{t+1} = w_t - \eta_t d_{\xi_t}  S^{\xi_t}_{u_t} \nabla   f(\hat{w}_t;\xi_t)$, we have 
$$ \|w_{t+1}-w_{*}\|^2 = \|w_t-w_{*}\|^2 - 2\eta_t \langle d_{\xi_t}  S^{\xi_t}_{u_t} \nabla   f(\hat{w}_t;\xi_t), (w_t-w_{*})\rangle + \eta_t^2 \| d_{\xi_t}  S^{\xi_t}_{u_t}  \nabla  f(\hat{w}_t;\xi_t)\|^2.$$
We now take expectations over $u_t$ and $\xi_t$ and use Lemma \ref{lem:expect}:
\begin{eqnarray*}
&& \mathbb{E}[\| w_{t+1} - w_{*} \|^2 | \mathcal{F}_{t}] \\
&\leq &
   \| w_{t} - w_{*} \|^2 - 2 \eta_t \langle \nabla F(\hat{w}_t)  , (w_{t} - w_{*}) \rangle + \eta^2_t D \mathbb{E}[\|\nabla f(\hat{w}_t; \xi_t)\|^2 | \mathcal{F}_{t} ] \\
&=&  \| w_{t} - w_{*} \|^2 - 2 \eta_t \langle \nabla F(\hat{w}_t)  , (w_{t} - \hat{w}_t) \rangle - 2 \eta_t \langle \nabla F(\hat{w}_t)  , (\hat{w}_t - w_{*}) \rangle  + \eta^2_t D \mathbb{E}[\|\nabla f(\hat{w}_t; \xi_t)\|^2 | \mathcal{F}_{t} ].
\end{eqnarray*}
By \eqref{eq:stronglyconvex_00} and \eqref{eq:Lsmooth}, we have
\begin{align*}
- \langle \nabla F(\hat{w}_t)  , (\hat{w}_t - w_{*}) \rangle & \leq - [ F(\hat{w}_t) - F(w_{*}) ] - \frac{\mu}{2} \| \hat{w}_t - w_{*} \|^2, \mbox{ and} \tagthis \label{eq_newlemma_011} \\
- \langle \nabla F(\hat{w}_t)  , (w_{t} - \hat{w}_t) \rangle & \leq F(\hat{w}_t) - F(w_{t}) + \frac{L}{2} \| \hat{w}_t - w_{t}  \|^2  \tagthis \label{eq_newlemma_021} 
\end{align*}

Thus, $\mathbb{E}[\| w_{t+1} - w_{*} \|^2 | \mathcal{F}_{t}]$ is at most
\begin{align*}
& \overset{\eqref{eq_newlemma_011}\eqref{eq_newlemma_021}}{\leq} \| w_{t} - w_{*} \|^2 + 2 \eta_t [ F(\hat{w}_t) - F(w_{t})  ] + L\eta_t \| \hat{w}_t - w_{t}  \|^2  - 2 \eta_t[ F(\hat{w}_t ) - F(w_{*}) ] - \mu \eta_t \| \hat{w}_t - w_{*} \|^2 \\ &\qquad + \eta_t^2 D \mathbb{E}[\|\nabla f(\hat{w}_t; \xi_t)\|^2 | \mathcal{F}_{t} ] \\
&= \| w_{t} - w_{*} \|^2 - 2 \eta_t [ F(w_{t}) - F(w_{*})  ] + L\eta_t \| \hat{w}_t - w_{t}  \|^2  - \mu \eta_t \| \hat{w}_t - w_{*} \|^2 + \eta^2_t D \mathbb{E}[\|\nabla f(\hat{w}_t; \xi_t)\|^2 | \mathcal{F}_{t} ].
\end{align*}

Since \begin{align*}
- \|\hat{w}_t - w_{*}\|^2 = - \|(w_{t} - w_{*}) - (w_{t} - \hat{w}_t) \|^2 \overset{\eqref{eq_aaa002}}{\leq} - \frac{1}{2} \| w_{t} - w_{*} \|^2 + \| w_{t} - \hat{w}_t \|^2, %\tagthis \label{eq_newlemma_03}
\end{align*}
$\mathbb{E}[\| w_{t+1} - w_{*} \|^2 | \mathcal{F}_{t}, \sigma_t] $ is at most 
$$
(1- \frac{\mu \eta_t}{2}) \| w_{t} - w_{*} \|^2 - 2 \eta_t [ F(w_{t}) - F(w_{*})  ] + (L+\mu)\eta_t \| \hat{w}_t- w_{t}  \|^2 %- \frac{\mu \eta_t}{2} \| w_{t} - w_{*} \|^2 
% + \mu \eta_t \| w_{t} - \hat{w}_t  \|^2  
 + \eta^2_t D \mathbb{E}[\|\nabla f(\hat{w}_t; \xi_t)\|^2 | \mathcal{F}_{t} ].
$$

We now use $\| a \|^2 = \| a - b + b \|^2 \leq 2 \| a - b\|^2 + 2 \|b\|^2$ for $\mathbb{E}[\|\nabla f(\hat{w}_t; \xi_t)\|^2 | \mathcal{F}_{t} ]$ to obtain
\begin{equation}
\mathbb{E}[\|\nabla f(\hat{w}_t; \xi_t)\|^2 | \mathcal{F}_{t}] \leq 
 2 \mathbb{E}[\|\nabla f(\hat{w}_t; \xi_t)- \nabla f(w_{t}; \xi_t) \|^2 | \mathcal{F}_{t}] + 2 \mathbb{E}[\|\nabla f(w_{t}; \xi_t)\|^2 | \mathcal{F}_{t} ]. \label{ineq:Estuff}
%\mathbb{E}[\| w_{t+1} - w_{*} \|^2 | \mathcal{F}_{t}]&\leq \| w_{t} - w_{*} \|^2 - 2 \eta_t [ F(w_{t}) - F(w_{*})  ] + (L + \mu) \eta_t \| \hat{w}_t - w_{t}  \|^2  - \frac{\mu \eta_t}{2} \| w_{t} - w_{*} \|^2  \\ &\qquad + 2 \eta^2_tD \mathbb{E}[\|\nabla f(\hat{w}_t; \xi_t)- \nabla f(w_{t}; \xi_t) \|^2 | \mathcal{F}_{t} ] + 2 \eta^2_t D \mathbb{E}[\|\nabla f(w_{t}; \xi_t)\|^2 | \mathcal{F}_{t} ].
\end{equation}
By Lemma \ref{lem_bounded_secondmoment_04}, we have
\begin{gather*}
\mathbb{E}[\|\nabla f(w_{t}; \xi_t)\|^2 | \mathcal{F}_{t}] \leq  4 L [ F(w_{t}) - F(w_{*}) ] + N. \tagthis \label{ineq:bounded_lemma3_hogwild}
\end{gather*} 
Applying \eqref{eq:Lsmooth_basic} twice\ gives 
$$\mathbb{E}[\|\nabla f(\hat{w}_t; \xi_t)- \nabla f(w_{t}; \xi_t) \|^2 | \mathcal{F}_{t},\sigma_t ]\leq L^2 \|\hat{w}_t - w_t \|^2$$
and  together with (\ref{ineq:Estuff}) and (\ref{ineq:bounded_lemma3_hogwild}) we obtain
$$\mathbb{E}[\|\nabla f(\hat{w}_t; \xi_t)\|^2 | \mathcal{F}_{t} ] \leq  2 L^2 \|\hat{w}_t - w_t \|^2 + 4 L [ F(w_{t}) - F(w_{*}) ] + N.$$
Plugging this into the previous derivation yields
 %  
 %   (i.e., $\| \nabla f(w;\xi) - \nabla f(w';\xi) \| \leq L \| w - w' \|$) for $ \mathbb{E}[\|\nabla f(\hat{w}_t; \xi_t)- \nabla f(w_{t}; \xi_t) \|^2 | \mathcal{F}_{t} ]$, \eqref{ineq:bounded_lemma3_hogwild} for $\mathbb{E}[\|\nabla f(w_{t}; \xi_t)\|^2 | \mathcal{F}_{t} ]$, yields 
\begin{align*}
\mathbb{E}[\| w_{t+1} - w_{*} \|^2 | \mathcal{F}_{t}]& \leq %\overset{\eqref{eq:Lsmooth_basic}\eqref{ineq:bounded_lemma3_hogwild}}{\leq} 
(1- \frac{\mu \eta_t}{2}) \| w_{t} - w_{*} \|^2 - 2 \eta_t [ F(w_{t}) - F(w_{*})  ] + (L + \mu) \eta_t \| \hat{w}_{t} - w_{t}  \|^2 %- \frac{\mu \eta_t}{2} \| w_{t} - w_{*} \|^2
 \\ &\qquad  + 2 L^2 \eta^2_t D \| \hat{w}_t - w_{t}  \|^2  + 8 L \eta^2_t D [F(w_{t}) - F(w_{*})] + 2\eta^2_t D N \\
&=(1- \frac{\mu \eta_t}{2}) \| w_{t} - w_{*} \|^2  + [(L+\mu)\eta_t + 2L^2\eta_t^2D] \| \hat{w}_t - w_{t}  \|^2  \\ &\qquad 
 - 2\eta_t( 1 - 4L\eta_tD) [F(w_{t}) - F(w_{*})] + 2\eta^2_t D N. 
\end{align*}
Since $\eta_t \leq \frac{1}{4LD}$, $-2\eta_t(1-4L\eta_tD) [F(w_{t}) - F(w_{*})]\leq 0$ (we can get a negative upper bound by applying strong convexity but this will not improve the asymptotic behavior of the convergence rate in our main result although it would improve the constant of the leading term making the final bound applied to SGD closer to the bound of Theorem \ref{thm_res_sublinear_new_02} for SGD),
$$\mathbb{E}[\| w_{t+1} - w_{*} \|^2 | \mathcal{F}_{t}] \leq \left(1 - \frac{\mu \eta_t}{2} \right) \| w_{t} - w_{*} \|^2 + [(L+\mu)\eta_t + 2L^2\eta_t^2D] \| \hat{w}_t - w_{t}  \|^2 + 2\eta^2_tD N$$ 
and this concludes the proof.
\end{proof} 

Assume %$\rho_t\leq R$ and 
$0 < \eta_t \leq \frac{1}{4LD}$  for all $t \geq 0$. Then,  after taking the full expectation of the inequality in  Lemma \ref{lemma:hogwild_1_new1}, we can
plug Lemma \ref{lemma:hogwild_21} into it which yields the recurrence
\begin{eqnarray} 
 \mathbb{E}[\| w_{t+1} - w_{*} \|^2]
&\leq& \left(1 - \frac{\mu \eta_t}{2} \right) \mathbb{E}[\| w_{t} - w_{*} \|^2] +  \nonumber  \\
&& [(L+\mu)\eta_t + 2L^2\eta_t^2D] (1+\sqrt{\Delta}\tau) D \sum_{j=t-\tau}^{t-1} \eta_j^2 (2L^2 \mathbb{E}[\| \hat{w}_j -w_{*} \|^2] +N) + 2\eta^2_tD N. \label{eq:rec} \end{eqnarray}
 This can be solved by using the next lemma. For completeness, we follow the convention that an empty product is equal to 1 and an empty sum is equal to 0, i.e., 
\begin{gather*}
\prod_{i=h}^k g_i = 1 \text{ and } \sum_{i=h}^k g_i =0 \text{ if } k<h. \tagthis \label{eq:rec00} 
\end{gather*}

\begin{lem}
\label{lemma:hogwild_recursive_form}
Let $Y_t, \beta_t$ and $\gamma_t$ be sequences such that $Y_{t+1} \leq \beta_t Y_t + \gamma_t$, for all $t\geq 0$. Then,
\begin{equation}
 \label{hogwild:recursive_form}
Y_{t+1} \leq (\sum_{i=0}^t[\prod_{j=i+1}^t \beta_j]\gamma_i)+(\prod_{j=0}^t\beta_j)Y_0. 
\end{equation}
\end{lem}

\begin{proof}
We prove the lemma by using induction. It is obvious that \eqref{hogwild:recursive_form} is true for $t=0$ because $Y_1\leq \beta_1Y_0+ \gamma_1$. Assume as induction hypothesis that \eqref{hogwild:recursive_form} is true for $t-1$. Since $Y_{t+1} \leq \beta_t Y_t + \gamma_t$,
\begin{align*}
Y_{t+1} &\leq \beta_t Y_t + \gamma_t \\
&\leq \beta_{t}[(\sum_{i=0}^{t-1}[\prod_{j=i+1}^{t-1} \beta_j]\gamma_i)+(\prod_{j=0}^{t-1}\beta_j)Y_0] + \gamma_{t} \\
&\overset{\eqref{eq:rec00}}{=} (\sum_{i=0}^{t-1} \beta_{t}[\prod_{j=i+1}^{t-1} \beta_j]\gamma_i)+\beta_{t}(\prod_{j=0}^{t-1}\beta_j)Y_0 + (\prod_{j=t+1}^{t}\beta_j)\gamma_{t}\\
&= [(\sum_{i=0}^{t-1} [\prod_{j=i+1}^{t} \beta_j]\gamma_i)+ (\prod_{j=t+1}^{t}\beta_j)\gamma_{t}] +(\prod_{j=0}^{t}\beta_j)Y_0 \\
&= (\sum_{i=0}^{t}[\prod_{j=i+1}^{t} \beta_j]\gamma_i)+(\prod_{j=0}^{t}\beta_j)Y_0.
\end{align*}
\end{proof} 

Applying the above lemma to (\ref{eq:rec}) will yield the following bound.
 
\begin{lem}\label{lemma:hogwild_4_new1} %Assume %$d_t\leq D$ and 
%$\rho_t\leq R$ and
Let $\eta_t = \frac{\alpha_t}{\mu(t+E)}$ with $4\leq \alpha_t \leq\alpha$ and $E = \max\{ 2\tau, \frac{4 L \alpha D}{\mu}\}$. Then, expectation $\mathbb{E}[\| w_{t+1} - w_{*} \|^2]$ is at most
\begin{align*}
 \frac{\alpha^2D }{\mu^2} \frac{1}{(t + E - 1)^2} \left(\sum_{i=1}^t
 \left[ 4 a_i (1+\sqrt{\Delta}\tau)   [ N\tau + 2L^2 \sum_{j=i-\tau}^{i-1} \mathbb{E}[\| \hat{w}_j -w_* \|^2 ] +
 2 N \right]
\right)+ \frac{(E+1)^2}{(t + E - 1)^2} \mathbb{E}[\| w_{0} - w_{*} \|^2],
% \sum_{i=1}^t \frac{\alpha^2}{\mu^2} \frac{1}{(t + E - 1)^2} \left[ 8R^2 a_i [D^2\mathbb{E}[m_{i-1}] + \frac{RN}{2}] + 2 N D\right] + \frac{E^2}{(t+E-1)^2} \mathbb{E}[\| w_{1} - w_{*} \|^2].
\end{align*} 
where $a_i = (L+\mu)\eta_i + 2L^2\eta_i^2D$.
\end{lem}

\begin{proof}
%We  bound (\ref{eq:rec}) in a coarse way (
Notice that we may use   (\ref{eq:rec}) because  $\eta_t \leq \frac{1}{4LD}$ follows from $\eta_t = \frac{\alpha_t}{\mu(t+E)} \leq \frac{\alpha}{\mu(t+E)}$ combined with $E\geq \frac{4 L \alpha D}{\mu}$. From (\ref{eq:rec}) with $a_t = (L+\mu)\eta_t + 2L^2\eta_t^2D$  and $\eta_t$ being decreasing in $t$ we infer 
% and in a later lemma we will revisit (\ref{eq:rec}) and plug in what we will learn in this lemma: We use
\begin{eqnarray*}
&&  \mathbb{E}[\| w_{t+1} - w_{*} \|^2] \\
&\leq& \left(1 - \frac{\mu \eta_t}{2} \right) \mathbb{E}[\| w_{t} - w_{*} \|^2] + a_t (1+\sqrt{\Delta}\tau) D \eta_{t-\tau}^2 \sum_{j=t-\tau}^{t-1} (2L^2 \mathbb{E}[\| \hat{w}_j -w_{*} \|^2] +N) + 2\eta^2_tD N \\
 &=&
 \left(1 - \frac{\mu \eta_t}{2} \right) \mathbb{E}[\| w_{t} - w_{*} \|^2] + a_t (1+\sqrt{\Delta}\tau) D \eta_{t-\tau}^2 [ N\tau + 2L^2 \sum_{j=t-\tau}^{t-1} \mathbb{E}[\| \hat{w}_j -w_{*} \|^2 ] + 2\eta^2_tD N.
 \end{eqnarray*}
 
 Since $E\geq 2\tau$,  $\frac{1}{t-\tau+E}\leq \frac{2}{t+E}$. Hence, together with  $\eta_{t-\tau} = \frac{\alpha_{t-\tau}}{\mu(t-\tau+E)} \leq \frac{\alpha}{\mu(t-\tau+E)}$ we have
 \begin{equation} \eta_{t-\tau}^2 \leq \frac{4\alpha^2}{\mu^2}\frac{1}{(t+E)^2}. \label{etatR} \end{equation}
 This translates the above bound into
 \begin{align*}
\mathbb{E}[\| w_{t+1} - w_{*} \|^2] & \leq \beta_t \mathbb{E}[\| w_{t} - w_{*} \|^2] + \gamma_t,
\end{align*}
for 
\begin{align*}
\beta_t &= 1 - \frac{\mu \eta_t}{2}, \\
\gamma_t &=  4 a_t (1+\sqrt{\Delta}\tau)  D  \frac{\alpha^2}{\mu^2}\frac{1}{(t+E)^2} [ N\tau + 2L^2 \sum_{j=t-\tau}^{t-1} \mathbb{E}[\| \hat{w}_j -w_{*} \|^2 ] + 2\eta^2_tD N, where \\
a_t &=  (L+\mu)\eta_t + 2L^2\eta_t^2D.
\end{align*}

Application of Lemma \ref{lemma:hogwild_recursive_form} for $Y_{t+1} = \mathbb{E}[\| w_{t+1} - w_{*} \|^2]$ and $Y_{t} = \mathbb{E}[\| w_{t} - w_{*} \|^2]$ gives
\begin{align*}
\mathbb{E}[\| w_{t+1} - w_{*} \|^2] \leq \left(\sum_{i=0}^t \left[\prod_{j=i+1}^t \left(1 - \frac{\mu \eta_j}{2} \right) \right]\gamma_i \right)+\left(\prod_{j=0}^t \left(1 - \frac{\mu \eta_j}{2} \right) \right) \mathbb{E}[\| w_{0} - w_{*} \|^2]. 
\end{align*}
In order to analyze this formula, since $\eta_j = \frac{\alpha_j}{\mu(j + E)}$ with $\alpha_j\geq 4$, we have
\begin{align*}
1 - \frac{\mu \eta_j}{2} &= 1 - \frac{\alpha_j}{2(j+E)} \leq 1 - \frac{2}{j +E}, 
\end{align*}
%\begin{align*}
%1 - \frac{\mu \eta_j}{2} &= 1 - \frac{\alpha}{2(j+E)} \leq 1 - \frac{1}{j +E} = \frac{j + E - 1}{j + E}, 
%\end{align*}
Hence (we can also use $1-x\leq e^{-x}$ which leads to similar results and can be used to show that our choice for $\eta_t$ leads to the tightest convergence rates in our framework), 
%\begin{align*}
%\prod_{j=1}^t \left(1 - \frac{\mu \eta_j}{2} \right) &\leq \prod_{j=1}^t \left( \frac{j + E - 1}{j + E} \right) = \frac{E}{t + E}, \tagthis \label{eq_newlemma02_02} \\
%\prod_{j=i+1}^t \left(1 - \frac{\mu \eta_j}{2} \right) &\leq \prod_{j=i+1}^t \left( \frac{j + E - 1}{j + E} \right) = \frac{i + E}{t + E}. \tagthis \label{eq_newlemma02_03}
%\end{align*}
\begin{align*}
\prod_{j=i}^t \left(1 - \frac{\mu \eta_j}{2} \right) &\leq \prod_{j=i}^t \left(1-\frac{2}{j+E}\right) =  \prod_{j=i}^t \frac{j+E-2}{j+E}\\
 &=\frac{i+E-2}{i+E}\frac{i+E-1}{i+E+1}\frac{i+E}{i+E+2}\frac{i+E+1}{i+E+3}\dotsc\frac{t+E-3}{t+E-1}\frac{t+E-2}{t+E} \\
 &=\frac{(i+E-2)(i+E-1)}{(t+E-1)(t+E)} \leq \frac{(i+E-1)^2}{(t+E-1)(t+E)} \leq \frac{(i+E)^2}{(t+E-1)^2}   .  %\tagthis \label{eq:hogwild_eq02} \\
 %\text{ and } \prod_{j=1}^t (1-\frac{2}{j+E}) &\overset{\eqref{eq:hogwild_eq01}}{\leq}  \frac{E(E-1)}{(t+E-1)(t+E)} \leq \frac{E^2}{(t+E-1)(t+E)} \leq \frac{E^2}{(t+E-1)^2}. %  \tagthis \label{eq:hogwild_eq03} 
\end{align*}

%%%%%%%%% IT ALL WORKS: beta_t a bit better -- or just alpha larger; R(t)<=t gives E(t)<=2t gives with alhha lower bound 4*3 -- yep works (alpha new lower bound, i/(i+E)<=1 etc. %%% 

From this calculation  we infer that
\begin{align*}
\mathbb{E}[\| w_{t+1} - w_{*} \|^2] \leq \left(\sum_{i=0}^t \left[ \frac{(i + E)^2}{(t + E - 1)^2} \right]\gamma_i \right)+ \frac{(E+1)^2}{(t + E - 1)^2} \mathbb{E}[\| w_{0} - w_{*} \|^2]. \tagthis \label{eq_newlemma02_041}
\end{align*}

Now, we substitute  $\eta_i \leq \frac{\alpha}{\mu(i + E)}$ in $\gamma_i$ and compute
\begin{align*}
&  \frac{(i + E)^2}{(t + E - 1)^2} \gamma_i \\
&=  \frac{(i + E)^2}{(t + E - 1)^2}
 4 a_i (1+\sqrt{\Delta}\tau)  D  \frac{\alpha^2}{\mu^2}\frac{1}{(i+E)^2} [ N\tau + 2L^2 \sum_{j=i-\tau}^{i-1} \mathbb{E}[\| \hat{w}_j -w_* \|^2 ] 
   +  \frac{(i + E)^2}{(t + E - 1)^2} 2 N D \frac{\alpha^2}{\mu^2(i + E)^2} \\
 &= \frac{\alpha^2D }{\mu^2} \frac{1}{(t + E - 1)^2} \left[ 4 a_i (1+\sqrt{\Delta}\tau)   [ N\tau + 2L^2 \sum_{j=i-\tau}^{i-1} \mathbb{E}[\| \hat{w}_j -w_{*} \|^2 ] +
 2 N \right]. 
\end{align*}
Substituting this in \eqref{eq_newlemma02_041} proves the lemma. 
\end{proof}

As an immediate corollary we can apply the inequality $\|a+b\|^2\leq 2\|a\|^2+2\|b\|^2$ to $\mathbb{E}[\|\hat{w}_{t+1} - w_* \|^2]$ to obtain 
\begin{equation}
\label{eq:eq11111}
\mathbb{E}[\|\hat{w}_{t+1} - w_* \|^2] \leq 2\mathbb{E}[\|\hat{w}_{t+1} - w_{t+1} \|^2]
+2\mathbb{E}[\|w_{t+1} - w_*\|^2],
\end{equation}
which in turn can be bounded by the previous lemma together with Lemma \ref{lemma:hogwild_21}:
\begin{eqnarray*}
\mathbb{E}[\|\hat{w}_{t+1} - w_* \|^2]
& \leq& 2(1+\sqrt{\Delta}\tau)  D \sum_{j=t+1-\tau}^{t} \eta_j^2 (2L^2 \mathbb{E}[\| \hat{w}_j -w_{*} \|^2] +N) + \\
&& 2\frac{\alpha^2D }{\mu^2} \frac{1}{(t + E - 1)^2} \left(\sum_{i=1}^t
 \left[ 4 a_i (1+\sqrt{\Delta}\tau)   [ N\tau + 2L^2 \sum_{j=i-\tau}^{i-1} \mathbb{E}[\| \hat{w}_j -w_{*} \|^2 ] +
 2 N \right]
\right)  + \\
&& \frac{(E+1)^2}{(t + E - 1)^2} \mathbb{E}[\| w_{0} - w_{*} \|^2]. 
\end{eqnarray*}

Now assume a decreasing sequence $Z_t$ for which we want to prove that $\mathbb{E}[\| \hat{w}_{t} - w_{*} \|^2]\leq Z_t$ by induction in $t$. Then, the above bound can be used together with the property that $Z_t$ and $\eta_t$ are decreasing in $t$ to show
$$\sum_{j=t+1-\tau}^{t} \eta_j^2 (2L^2 \mathbb{E}[\| \hat{w}_j -w_* \|^2] +N) \leq \tau  \eta_{t-\tau}^2 (2L^2 Z_{t+1-\tau} +N) \leq
4\tau \frac{\alpha^2}{\mu^2}\frac{1}{(t+E-1)^2}  (2L^2 Z_{t+1-\tau} +N),$$
where the last inequality follows from (\ref{etatR}), and
$$\sum_{j=i-\tau}^{i-1} \mathbb{E}[\| \hat{w}_j -w_{*} \|^2 ]\leq \tau Z_{i-\tau}.$$
From these inequalities we infer
\begin{eqnarray}
\mathbb{E}[\|\hat{w}_{t+1} - w_* \|^2]
& \leq& 8(1+\sqrt{\Delta}\tau) \tau D \frac{\alpha^2}{\mu^2}\frac{1}{(t+E-1)^2}  (2L^2 Z_{t+1-\tau} +N) + \nonumber \\
&& 2\frac{\alpha^2D }{\mu^2} \frac{1}{(t + E - 1)^2} \left(\sum_{i=1}^t
 \left[ 4 a_i (1+\sqrt{\Delta}\tau)  [ N\tau + 2L^2 \tau Z_{i-\tau}] +
 2 N \right]
\right)  + \nonumber \\
&& \frac{(E+1)^2}{(t + E - 1)^2} \mathbb{E}[\| w_{0} - w_{*} \|^2] . \label{bestbound}
%&\leq& 2R^2 D  \eta_{t-R}^2 (2L^2 Z_{t+1-R} +N)
\end{eqnarray}

Even if we assume a constant $Z\geq Z_0\geq Z_1\geq Z_2 \geq \ldots$, we can get a first bound on the convergence rate of vectors $\hat{w}^t$: Substituting $Z$ gives
\begin{eqnarray}
\mathbb{E}[\|\hat{w}_{t+1} - w_* \|^2]
& \leq& 8(1+\sqrt{\Delta}\tau) \tau D \frac{\alpha^2}{\mu^2}\frac{1}{(t+E-1)^2}  (2L^2 Z +N) + \nonumber \\
&& 2\frac{\alpha^2D }{\mu^2} \frac{1}{(t + E - 1)^2} \left(\sum_{i=1}^t
 \left[ 4 a_i (1+\sqrt{\Delta}\tau)   [ N\tau + 2L^2 \tau Z] +
 2 N \right]
\right)  + \nonumber \\
&& \frac{(E+1)^2}{(t + E - 1)^2} \mathbb{E}[\| w_{0} - w_{*} \|^2] . \label{eqforGen}
%&\leq& 2R^2 D  \eta_{t-R}^2 (2L^2 Z_{t+1-R} +N)
\end{eqnarray}

Since $a_i = (L+\mu)\eta_i + 2L^2\eta_i^2 D$ and $\eta_i \leq \frac{\alpha}{\mu(i + E)}$, we have
\begin{align*}
\sum_{i=1}^t a_i &= (L+\mu) \sum_{i=1}^t \eta_i + 2 L^2 D \sum_{i=1}^t \eta_i^2 \\
&\leq (L+\mu) \sum_{i=1}^t \frac{\alpha}{\mu(i + E)} + 2 L^2 D\sum_{i=1}^t \frac{\alpha^2}{\mu^2(i + E)^2} \\
&\leq \frac{(L+\mu)\alpha}{\mu} \sum_{i=1}^t \frac{1}{i} + \frac{2 L^2 \alpha^2D}{\mu^2} \sum_{i=1}^t \frac{1}{i^2} \\
&\leq \frac{(L+\mu)\alpha}{\mu} (1 + \ln t) + \frac{ L^2 \alpha^2 D\pi^2}{3 \mu^2}, \tagthis \label{eq_newlemma03_0211}
\end{align*}
where the last inequality is a property of the harmonic sequence $\sum_{i=1}^t \frac{1}{i}\leq 1 + \ln t$ and $\sum_{i=1}^t \frac{1}{i^2} \leq \sum_{i=1}^{\infty} \frac{1}{i^2} = \frac{\pi^2}{6}$. 

Substituting (\ref{eq_newlemma03_0211}) in (\ref{eqforGen}) and collecting terms yields
\begin{eqnarray}
&& \mathbb{E}[\|\hat{w}_{t+1} - w_* \|^2] \nonumber \\
& \leq&
2\frac{\alpha^2D }{\mu^2} \frac{1}{(t + E - 1)^2} \left(2N t + 
4(1+\sqrt{\Delta}\tau) \tau [N+2L^2Z] \left\{ \frac{(L+\mu)\alpha}{\mu} (1 + \ln t) + \frac{ L^2 \alpha^2 D\pi^2}{3 \mu^2 +1} \right\}
\right)  + \nonumber  \\
&& \frac{(E+1)^2}{(t + E - 1)^2} \mathbb{E}[\| w_{0} - w_{*} \|^2] .\label{refbound}
%&\leq& 2R^2 D  \eta_{t-R}^2 (2L^2 Z_{t+1-R} +N)
\end{eqnarray}
Notice that the asymptotic behavior in $t$ is dominated by the term
$$ \frac{4\alpha^2D N}{\mu^2} \frac{t}{(t + E - 1)^2}.$$
If we define $Z_{t+1}$ to be the right hand side of (\ref{refbound}) and observe that this $Z_{t+1}$ is decreasing and a constant $Z$ exists (since the terms with $Z$ decrease much faster in $t$ compared to the dominating term), then  this $Z_{t+1}$ satisfies the derivations done above  and a proof by induction can be completed. 

Our derivations prove our main result: The expected convergence rate of read vectors is 
 \begin{eqnarray*}
\mathbb{E}[\|\hat{w}_{t+1} - w_* \|^2]
& \leq&  \frac{4\alpha^2D N}{\mu^2} \frac{t}{(t + E - 1)^2} + O\left(\frac{\ln t}{(t+E-1)^{2}}\right).
\end{eqnarray*}
We can use this result in Lemma \ref{lemma:hogwild_4_new1} in order to show that
%$\mathbb{E}[\|w_{t+1} - w_* \|^2]$ has the same  convergence rate) and
  the expected convergence rate $\mathbb{E}[\|w_{t+1} - w_* \|^2]$ satisfies the same bound. 

We remind the reader, that in the $(t+1)$-th iteration at most $\leq \lceil |D_{\xi_t}|/D \rceil$ vector positions are updated. Therefore the expected number of single vector entry updates is at most $\bar{\Delta}_D /D$.

%We use  $t\geq t'/K$, see  (\ref{eqK}), to translate the convergence rate in terms of the total number $t'$ of vector position updates that have happend so far (we also use $1/(t+E-1)\leq 1/t$ and $D=\lceil \Delta/K \rceil$):

%Our derivations prove our main result (we can use this result in Lemma \ref{lemma:hogwild_4_new1} in order to show that
%$\mathbb{E}[\|w_{t+1} - w_* \|^2]$ has the same  convergence rate):

\textbf{Theorem~\ref{theorem:Hogwild_newnew1}}. 
%\begin{thm}
%\label{theorem:Hogwild_newnew}
 \textit{Suppose Assumptions \ref{ass_stronglyconvex}, \ref{ass_smooth}, \ref{ass_convex} and \ref{ass_tau}  and  consider Algorithm~\ref{HogWildAlgorithm}. Let  $\eta_t = \frac{\alpha_t}{\mu(t+E)}$ with $4\leq \alpha_t \leq\alpha$ and $E = \max\{ 2\tau, \frac{4 L \alpha D}{\mu}\}$. Then, $t' = t \bar{\Delta}_D /D$ is the expected number of single vector entry updates after $t$ iterations and
 expectations
 $\mathbb{E}[\|\hat{w}_{t} - w_* \|^2]$ and $\mathbb{E}[\|w_{t} - w_* \|^2]$ are at most
  \begin{eqnarray*}
&&  \frac{4\alpha^2D N}{\mu^2} \frac{t}{(t + E - 1)^2} + O\left(\frac{\ln t}{(t+E-1)^{2}}\right).
%\frac{4\alpha^2 \bar{\Delta}_D N}{\mu^2} \frac{1}{t'} + O\left(\frac{\ln t'}{t'^{2}}\right).
\end{eqnarray*} 
}

\subsection{Convergence without Convexity of Component Functions}
 
 For the non-convex case, $L$ in (\ref{LN}) must be replaced by $L\kappa$ and as a result $L^2$ in Lemma \ref{lemma:hogwild_21} must be replaced by $L^2\kappa$.  Also $L$ in (\ref{ineq:bounded_lemma3_hogwild}) must be replaced by $L\kappa$. We now require that $\eta_t \leq \frac{1}{4L\kappa D}$ so that $-2\eta_t(1-4L\kappa \eta_tD) [F(w_{t}) - F(w_{*})]\leq 0$. This leads to Lemma \ref{lemma:hogwild_1_new1} where no changes are needed except requiring  $\eta_t\leq \frac{1}{4L\kappa D}$. The changes in Lemmas \ref{lemma:hogwild_21} and \ref{lemma:hogwild_1_new1}  lead to a Lemma \ref{lemma:hogwild_4_new1} where we require $E\geq \frac{4L\kappa \alpha D}{\mu}$ and where in the bound of the expectation $L^2$ must be replaced by $L^2\kappa$. This perculates through to inequality (\ref{refbound})  with a similar change finally leading to Theorem \ref{thm_6}, i.e., Theorem \ref{theorem:Hogwild_newnew1} where we only need to strengthen the condition on $E$ to $E\geq \frac{4L\kappa \alpha D}{\mu}$ in order to remove Assumption \ref{ass_convex}.
 
 \subsection{Sensitivity to $\tau$}
 
 What about the upper bound's sensitivity with respect to $\tau$?
 Suppose $\tau$ is not a constant but an increasing function of $t$, which also makes $E$ a function of $t$:
 $$ \frac{2 L \alpha D}{\mu}\leq \tau(t) \leq t  \mbox{ and } E(t)=2\tau(t).$$
 In order to obtain a similar theorem we  increase the lower bound on $\alpha_t$ to 
 $$ 12\leq \alpha_t \leq \alpha.$$
 This allows us to modify the proof of Lemma \ref{lemma:hogwild_4_new1} where we analyse the product
 $$ \prod_{j=i}^t \left(1 - \frac{\mu \eta_j}{2} \right).$$
 Since $\alpha_j\geq 12$ and $E(j)=2\tau(j)\leq 2j$,
 $$ 1-\frac{\mu \eta_j}{2} = 1 - \frac{\alpha_j}{2(j+E(j))}\leq 1- \frac{12}{2(j+2j)}=1-\frac{2}{j}\leq 1-\frac{2}{j+1}.$$
 The remaining part of the proof of Lemma \ref{lemma:hogwild_4_new1} continues as before where constant $E$ in the proof is replaced by $1$. This yields
 instead of (\ref{eq_newlemma02_041})  \begin{align*}
\mathbb{E}[\| w_{t+1} - w_{*} \|^2] \leq \left(\sum_{i=1}^t \left[ \frac{(i+1)^2}{t^2} \right]\gamma_i \right)+ \frac{4}{t^2} \mathbb{E}[\| w_{0} - w_{*} \|^2]. 
\end{align*}
We again substitute  $\eta_i \leq \frac{\alpha}{\mu(i + E(i))}$ in $\gamma_i$, realize that $\frac{(i+1)}{(i+E(i))}\leq 1$, and compute
\begin{align*}
& \frac{(i + 1)^2}{t^2} \gamma_i \\
 &=  \frac{(i + 1)^2}{t^2}
 4 a_i (1+\sqrt{\Delta}\tau(i)) D  \frac{\alpha^2}{\mu^2}\frac{1}{(i+E(i))^2} [ N\tau(i) + 2L^2 \sum_{j=i-\tau(i)}^{i-1} \mathbb{E}[\| \hat{w}_j -w_* \|^2 ] 
   +  \frac{(i + 1)^2}{t^2} 2 N D \frac{\alpha^2}{\mu^2(i + E(i))^2} \\
 &\leq \frac{\alpha^2D }{\mu^2} \frac{1}{t^2} \left[ 4 a_i (1+\sqrt{\Delta}\tau(i))  [ N\tau(i) + 2L^2 \sum_{j=i-\tau(i)}^{i-1} \mathbb{E}[\| \hat{w}_j -w_{*} \|^2 ] +
 2 N \right]. 
\end{align*}
This gives a new Lemma  \ref{lemma:hogwild_4_new1}:

\begin{lem}\label{lemma:hogwild_4_new1_Gen} Assume  $\frac{2 L \alpha D}{\mu}\leq \tau(t) \leq t$ with $\tau(t)$ monotonic increasing.
Let $\eta_t = \frac{\alpha_t}{\mu(t+E(t))}$ with $12\leq \alpha_t \leq\alpha$ and $E(t) = 2\tau(t)$. Then, expectation $\mathbb{E}[\| w_{t+1} - w_{*} \|^2]$ is at most
\begin{align*}
 \frac{\alpha^2D }{\mu^2} \frac{1}{t^2} \left(\sum_{i=1}^t
 \left[ 4 a_i (1+\sqrt{\Delta}\tau(i))  [ N\tau(i) + 2L^2 \sum_{j=i-\tau(i)}^{i-1} \mathbb{E}[\| \hat{w}_j -w_* \|^2 ] +
 2 N \right]
\right)+ \frac{4}{t^2} \mathbb{E}[\| w_{0} - w_{*} \|^2],
% \sum_{i=1}^t \frac{\alpha^2}{\mu^2} \frac{1}{(t + E - 1)^2} \left[ 8R^2 a_i [D^2\mathbb{E}[m_{i-1}] + \frac{RN}{2}] + 2 N D\right] + \frac{E^2}{(t+E-1)^2} \mathbb{E}[\| w_{1} - w_{*} \|^2].
\end{align*} 
where $a_i = (L+\mu)\eta_i + 2L^2\eta_i^2D$.
\end{lem}

Now we can continue the same analysis that led to Theorem \ref{theorem:Hogwild_newnew1} and conclude that there exists a constant $Z$ such that, see (\ref{eqforGen}),
\begin{eqnarray}
\mathbb{E}[\|\hat{w}_{t+1} - w_* \|^2]
& \leq& 8 (1+\sqrt{\Delta}\tau(t))\tau(t) D \frac{\alpha^2}{\mu^2}\frac{1}{t^2}  (2L^2 Z +N) + \nonumber \\
&& 2\frac{\alpha^2D }{\mu^2} \frac{1}{t^2} \left(\sum_{i=1}^t
 \left[ 4 a_i  (1+\sqrt{\Delta}\tau(i))  [ N\tau(i) + 2L^2 \tau(i) Z] +
 2 N \right]
\right)  + \nonumber \\
&& \frac{4}{t^2} \mathbb{E}[\| w_{0} - w_{*} \|^2] . \label{eqderGen1}
%&\leq& 2R^2 D  \eta_{t-R}^2 (2L^2 Z_{t+1-R} +N)
\end{eqnarray}

Let us assume
\begin{equation}
\tau(t)\leq \sqrt{t \cdot L(t)}, \label{eqR1}
\end{equation}
where 
$$L(t)= \frac{1}{\ln t} - \frac{1}{(\ln t)^2}$$
which has the property that the derivative of $t/(\ln t)$ is equal to $L(t)$.
Now we observe
\begin{eqnarray*}
 \sum_{i=1}^t a_i \tau(i)^2 &=& \sum_{i=1}^t [(L+\mu)\eta_i + 2L^2\eta_i^2D]\tau(i)^2
\leq  \sum_{i=1}^t [(L+\mu)\frac{\alpha}{\mu i} + 2L^2\frac{\alpha^2}{\mu^2 i^2}D]\cdot i L(i) \\
&=& \frac{(L+\mu)\alpha}{\mu}  \sum_{i=1}^t  L(i) + O(\ln t)
= \frac{(L+\mu)\alpha}{\mu}  \frac{t}{\ln t} + O(\ln t) 
\end{eqnarray*}
and
\begin{eqnarray*}
 \sum_{i=1}^t a_i \tau(i) &=& \sum_{i=1}^t [(L+\mu)\eta_i + 2L^2\eta_i^2D]\tau(i)
\leq  \sum_{i=1}^t [(L+\mu)\frac{\alpha}{\mu i} + 2L^2\frac{\alpha^2}{\mu^2 i^2}D]\cdot \sqrt{i}) \\
&=& O(\sum_{i=1}^t  \frac{1}{\sqrt{i}})
=  O(\sqrt{t}) .
\end{eqnarray*}

%x=e^y
%\int_{x=1}^t \frac{1}{\ln x} dx = \int_{y=0}^{\ln t} \frac{e^y}{y} dy = e^y \ln y 
Substituting both inequalities in  (\ref{eqderGen1}) gives
\begin{eqnarray}
\mathbb{E}[\|\hat{w}_{t+1} - w_* \|^2]
& \leq& 8  (1+\sqrt{\Delta}\tau(t))\tau(t) D \frac{\alpha^2}{\mu^2}\frac{1}{t^2}  (2L^2 Z +N) + \nonumber \\
&& 2\frac{\alpha^2D }{\mu^2} \frac{1}{t^2} \left(2Nt + 4\sqrt{\Delta}[  \frac{(L+\mu)\alpha}{\mu} \frac{t}{\ln t} + O(\ln t) ]  [ N + 2L^2 Z] +O(\sqrt{t})
\right)  + \nonumber \\
&& \frac{4}{t^2} \mathbb{E}[\| w_{0} - w_{*} \|^2]  \nonumber \\
&\leq &
 2\frac{\alpha^2D }{\mu^2} \frac{1}{t^2} \left(2Nt + 4\sqrt{\Delta}[ (1+ \frac{(L+\mu)\alpha}{\mu}) \frac{t}{\ln t} + O(\ln t) ]  [ N + 2L^2 Z] +O(\sqrt{t})
\right)  + \nonumber \\
&& \frac{4}{t^2} \mathbb{E}[\| w_{0} - w_{*} \|^2]   
    \label{ineqterms}
%&\leq& 2R^2 D  \eta_{t-R}^2 (2L^2 Z_{t+1-R} +N)
\end{eqnarray}
Again we define $Z_{t+1}$ as the right hand side of this inequality.  Notice that $Z_t = O(1/t)$, since the above derivation
 proves
$$ \mathbb{E}[\|\hat{w}_{t+1} - w_* \|^2] \leq  \frac{4\alpha^2D N}{\mu^2} \frac{1}{t} + O(\frac{1}{t\ln t}).$$

Summarizing we have the following main lemma:

\begin{lem} \label{lemtau}
Let Assumptions \ref{ass_stronglyconvex}, \ref{ass_smooth}, \ref{ass_convex} and \ref{ass_tau} hold and  consider Algorithm~\ref{HogWildAlgorithm}. 
Assume  $\frac{2 L \alpha D}{\mu}\leq \tau(t)\leq \sqrt{t \cdot L(t)}$ with $\tau(t)$ monotonic increasing.
Let  $\eta_t = \frac{\alpha_t}{\mu(t+2\tau(t))}$ with $12\leq \alpha_t \leq\alpha$. Then,
 the expected convergence rate of read vectors is 
 \begin{eqnarray*}
\mathbb{E}[\|\hat{w}_{t+1} - w_* \|^2]
& \leq& 
 \frac{4\alpha^2D N}{\mu^2} \frac{1}{t} + O(\frac{1}{t\ln t}),
\end{eqnarray*}
 where $L(t)=\frac{1}{\ln t} - \frac{1}{(\ln t)^2}$. The expected convergence rate $\mathbb{E}[\|w_{t+1} - w_* \|^2]$ satisfies the same bound.
 \end{lem}

Notice that we can plug $Z_t = O(1/t)$ back into an equivalent of (\ref{bestbound}) where we may bound $Z_{i-\tau(i)}=O(1/(i-\tau(i))$ which replaces $Z$ in the second line of (\ref{eqforGen}). On careful examination this leads to a new upper bound (\ref{ineqterms}) where the $2L^2Z$ terms gets absorped in a higher order term.
This can be used to show that, for
$$ t\geq T_0 = \exp[ 2\sqrt{\Delta}(1+\frac{(L+\mu)\alpha}{\mu})],$$
the higher order terms that contain $\tau(t)$ (as defined above) are at most the leading term as given in Lemma \ref{lemtau}.

% 1 >= 2   sqrtD (1+a+La/mu)  /ln t
% t >= \exp{ 2\sqrt{\Delta}(1+\frac{(L+\mu)\alpha}{\mu})}

 Upper bound (\ref{ineqterms}) also  shows that, for  
$$ t \geq T_1 = \frac{\mu^2}{\alpha^2 N D}\|w_0-w_*\|^2,$$
the higher order term that contains $\|w_0-w_*\|^2$ is at most the leading term. 
%\input{hogwild_UCONN_New_20180204.tex}
%\input{hogwild_main.tex}
%\input{hogwild.tex}

%
%\newpage 
%$\ $
%\newpage 
%\part{THIS CONTAINS A VERY VERY OLD STUFF, DO NOT EVEN LOOK THERE :)}
%
%\newpage 
%
%\input{bin.tex} 

\end{document}